\documentclass{amsart}

\usepackage[latin1]{inputenc}
\usepackage{fancyhdr}
\usepackage{graphicx}
\usepackage{newlfont}
\usepackage{amssymb}
\usepackage{amsmath,amscd}
\usepackage{latexsym}
\usepackage{amsthm}
\usepackage[vcentermath]{youngtab}
\usepackage{multirow}
\usepackage{bbm}
\theoremstyle{plain}                    
\newtheorem{thm}{Theorem}[section]
\newtheorem{prop}[thm]{Proposition}                          

\newtheorem{cor}[thm]{Corollary}
\newtheorem{lem}[thm]{Lemma}
\newtheorem{conj}[thm]{Conjecture}
\theoremstyle{definition}
\newtheorem{defn}[thm]{Definition}
\newtheorem{rem}[thm]{Remark}
\newtheorem{exa}[thm]{Example}

\newcommand{\rinto}{\hookrightarrow}

\newcommand{\ronto}{\twoheadrightarrow}

\newcommand{\df}{\displaystyle\frac}
\newcommand{\seq}[1]{\left<#1\right>}   
\newcommand{\F}{\mathcal{F}}
\newcommand{\End}{\operatorname{End}}
\newcommand{\Hom}{\operatorname{Hom}}

\newcommand{\N}{\mathcal{N}}

\newcommand{\D}{\mathfrak{D}}

\newcommand{\CC}{\mathbb{C}}
\newcommand{\ZZ}{\mathbb{Z}}
\newcommand{\Fq}{\mathbb{F}_q}
\newcommand{\q}{\mathbf{q}}

\newcommand{\HH}{\mathcal{H}}

\newcommand{\gl}{\operatorname{GL}}

\newcommand{\pp}{\mathfrak{p}}

\newcommand{\Ind}{\operatorname{Ind}}
\newcommand{\oo}{\emptyset}

\newcommand{\R}{\mathcal{R}}
\newcommand{\card}{\operatorname{card}}
\newcommand{\T}{\mathcal{T}}
\newcommand{\Res}{\operatorname{Res}}

\newcommand{\ull}{\underline{\lambda}}
\newcommand{\umu}{\underline{\mu}}
\newcommand{\unu}{\underline{\nu}}
\newcommand{\II}{\mathcal{I}}
\newcommand{\ee}{\varepsilon}
\newcommand{\Mod}{\operatorname{Mod}}
\newcommand{\K}{\mathcal{K}}
\newcommand{\G}{\mathcal{G}}
\newcommand{\ttl}{\mathfrak{t}^{\ull}}
\newcommand{\kk}{\mathbbm{k}}
\newcommand{\st}{\mathfrak{s}}
\newcommand{\tu}{\mathfrak{t}}
\newcommand{\qq}{\mathfrak{q}}

\begin{document}
\title{The Mirabolic Hecke Algebra}
\author{Daniele Rosso}
\address{\newline
Daniele Rosso \newline 
University of Ottawa \newline 
Department of Mathematics and Statistics\newline
585 King Edward, Ottawa, ON, K1N 6N5, Canada }

\email{drosso@uottawa.ca}
\date{\today}
\begin{abstract}The Iwahori-Hecke algebra of the symmetric group is the convolution algebra of $\gl_n$-invariant functions on the variety of pairs of complete flags over a finite field. Considering convolution on the space of triples of two flags and a vector we obtain the mirabolic Hecke algebra $R_n$, which had originally been described by Solomon. In this paper we give a new presentation for $R_n$, which shows that it is a quotient of a cyclotomic Hecke algebra as defined by Ariki and Koike. From this we recover the results of Siegel about the representations of $R_n$. We use Jucys-Murphy elements to describe the center of $R_n$ and to give a $\mathfrak{gl}_\infty$-structure on the Grothendieck group of the category of its representations, giving `mirabolic' analogues of classical results about the Iwahori-Hecke algebra. We also outline a strategy towards a proof of the conjecture that the mirabolic Hecke algebra is a cellular algebra.\end{abstract} 
\maketitle
\section{Introduction}
\subsection{} The Iwahori-Hecke algebra $H_n$ of the symmetric group $S_n$ is an example of a convolution algebra. The basic setting for convolution is the following: we have a finite set $X$ and we take $E=\CC(X\times X)$ to be the vector space of all complex valued functions on $X\times X$. Then given two functions $f,g$ we define their convolution to be 
\begin{equation}\label{conv1}(f*g)(x,y)=\sum_{z\in X}f(x,z)g(z,y).\end{equation} 
This defines an associative product on $E$. 
If $G$ is a group acting on $X$, then we have the diagonal action of $G$ on $X\times X$ which induces a $G$-action on $E$. We can then consider the algebra $E^G\simeq\CC(G\backslash X\times X)$ of functions that are invariant under the group action, with the same convolution product of \eqref{conv1}. Let $G=\gl_n(\Fq)$, and $B$ the subgroup of upper triangular matrices. Then we can take $X=G/B$, the space of all complete flags in $\Fq^n$. From the Bruhat decomposition, it follows that the $G$-orbits on $G/B\times G/B$ are parametrized by the symmetric group $S_n$. As we will discuss in Section \ref{iwh}, for $X=G/B$, the resulting convolution algebra is the Iwahori-Hecke algebra of $S_n$.

The action of $G$ on the space $G/B\times G/B\times \Fq^n$ still has finitely many orbits. This is a special case of the result of Magyar, Weyman and Zelevinsky, which in \cite{MWZ} have classified all cases in which the $G$ action on triples of flags has finitely many orbits. We can then define a convolution product on the space of invariant functions $\CC(G/B\times G/B\times\Fq^n)^G$. We call the resulting algebra $R_n$ the \emph{mirabolic Hecke algebra}.

\subsection{} For a field $\kk$, the action of $\gl_n(\kk)$ on varieties of flags and pairs of flags is a classical topic of study. Recently, generalizations of these constructions have been appearing, including the extra data of a vector or a line. This is what we mean by the `mirabolic' setting. The name comes from the mirabolic subgroup $P\subset \gl_n(\kk)$, which is the subgroup that fixes a nonzero vector in $V=\kk^n$. This is because in general, for a $G$-variety $X$, the $P$-orbits on $X$ are obviously in a 1-1 correspondence with $G$-orbits on $X\times (V\setminus\{0\})$.
 
One case in which such a generalization arises is the following. If $G/B$ is the variety of complete flags in $V$, then it is interesting to study the action of $G$ on $G/B\times G/B\times V$. One reason why this is important is because $\D$-modules on $G/B\times G/B\times V$ are closely related to mirabolic character $\D$-modules. These are certain $\D$-modules on $G\times V$, which arise when studying the spherical trigonometric Cherednik algebra, see \cite{FG}.  

Another example is the work of Achar and Henderson (\cite{AH}) extending the nilpotent cone $\N\subset\End(V)$ to the `enhanced nilpotent cone' $\N\times V$. Here the group $G$ acts on $V$ in the obvious way and on $\N$ by conjugation. 
Their motivation was the work of S.Kato (\cite{K}), which had introduced the `exotic nilpotent cone', of which $\N\times V$ is a simplification. Kato uses the exotic nilpotent cone to establish an `exotic' Springer correspondence and to give a geometric construction of the affine Hecke algebra of type $C^{(1)}_n$. 


\subsection{} The focus of the present paper is the `mirabolic' Hecke algebra $R_n$. This had originally been defined, in different terms, by Solomon in \cite{So}, and its irreducible representations have been described in \cite{Si}. The structure of the same space, as an $H_n$-bimodule, has been studied by Travkin in \cite{T}. 

Most of the results contained here follow from the observation that $R_n$ is a quotient of a cyclotomic Hecke algebra, as defined by Ariki and Koike. This is proved in Section \ref{cycloha} by giving a new presentation for $R_n$. It is similar to the main result of \cite{HR} about the $q$-rook algebra. We discuss these similarities in Section \ref{qrook}. The rest of the paper is organised as follows. In Sections \ref{iwh} and \ref{mirha} we recall results about the Iwahori-Hecke algebra of $S_n$, and we introduce $R_n$. In Section \ref{jme} we discuss analogues of the Jucys-Murphy elements. We use them to describe the center of $R_n$ and to give the structure of a $\mathfrak{gl}_\infty$-representation on the Grothendieck groups of the categories of $R_n$-modules. We conclude in Section \ref{cellstruct} by conjecturing that $R_n$ is a cellular algebra and giving a strategy towards the proof. This is carried out for the case of $R_2$.
\subsection*{Acknowledgments} The author would like to thank Victor Ginzburg for suggesting the line of inquiry that led to this work and for his help and advice. He thanks the people with whom he had useful conversations about this subject, in particular Sam Evens, Eric Vasserot and Jonathan Sun. Thanks also go to Emily Norton who was extremely helpful by pointing out the reference \cite{HR}. Finally, he is grateful to the University of Chicago and the University of Ottawa for support.
\section{Iwahori-Hecke Algebra of $S_n$}\label{iwh}
The symmetric group $S_n$ is a Coxeter group, with simple reflections being the adjacent transpositions $s_i=(i\, i+1)$ for $i=1,\ldots,n-1$. We define the corresponding \emph{Iwahori-Hecke Algebra} $H_n$ as follows.
\begin{defn}$H_n$ is the $\CC[\q,\q^{-1}]$-algebra, with generators $T_i$, $i=1,\ldots,n-1$ satisfying the following relations:
\begin{align}
T_{i}^2  &=(\q-1)T_{i}+\q & i=1,\ldots,n-1 & & \notag \\
T_{i}T_{i+1}T_{i}& =T_{i+1}T_{i}T_{i+1} & i=1,\ldots,n-2 & & \label{relq}\\
T_{i}T_{j}& =T_{j}T_{i} & |i-j|\geq 2 & & \notag
\end{align}
\end{defn}
It is well known that $H_n$ is a free $\CC[\q,\q^{-1}]$-module of rank $n!$ with a basis is given by $\{T_w | w\in S_n\}$. Here $T_w=T_{i_1}\cdots T_{i_r}$, where $w=s_{i_1}\cdots s_{i_r}$ is a reducd expression of $w$. The elements $T_w$ are well defined because of the braid relations in $\eqref{relq}$.
\begin{defn}If $q\in\CC^{\times}$ is a nonzero complex number, we define the specialization 
$$\HH_n(q):=H_n/(\q-q)$$
which is an algebra over $\CC$. We also define the generic Hecke algebra by extending scalars to the field of fractions
$$\HH_n:=H_n\otimes_{\CC[\q,\q^{-1}]}\CC(\q).$$
\end{defn}
When $q$ is taken to be a power of a prime, we can realize $\HH_n(q)$ as a convolution algebra, as follows.
Let $\Fq$ be the finite field with $q$ elements, and $G=\gl_n(\Fq)$ the general linear group. We take $B\leq G$ to be the Borel subgroup of upper trangular matrices, then $G/B$ is canonically identified with the variety of complete flags in $V=\Fq^d$. We let $E=\CC(G/B\times G/B)$ be the vector space of all complex valued functions on $G/B\times G/B$. 
\begin{defn} Given two functions $f,g\in E$ we define their \emph{convolution} to be 
\begin{equation}\label{convol}f*g(F,F')=\sum_{H\in G/B}f(F,H)g(H,F').\end{equation}\end{defn}
This is an associative product, with identity element the characteristic function of the diagonal. Now, $G$ acts on $G/B$ by left multiplications and it acts diagonally on $G/B\times G/B$, hence it acts on $E$. It is easily checked that the convolution product  descends to a well defined product on $E^G\simeq\CC(G\backslash G/B\times G/B)$. This is the space of functions that are invariant under the diagonal group action or equivalently the space of functions that are constant on the $G$-orbits. By the Bruhat decomposition, the orbits in this case are given by elements of the symmetric group. This can be seen as follows. If we have two flags 
$$F=(0=F_0\subset\ldots\subset F_{n-1}\subset F_n=V);\quad F'=(0=F'_0\subset\ldots\subset F'_{n-1}\subset F'_n=V)$$
we define their \emph{relative position} to be the $n\times n$ permutation matrix $w=w(F,F')$ with entries given by:
\begin{equation}\label{rpos}w(F,F')_{ij}=\dim\left(\frac{F_i\cap F'_j}{F_i\cap F'_{j-1}+F_{i-1}\cap F'_j}\right).\end{equation}

Let $\mathbf{T}_w$  be the characteristic function of the $G$-orbit of all pairs of flags in relative position $w$, in particular $\mathbf{T}_e$ is the identity in $E^G$.
\begin{prop}\label{isoconv} The map $\mathbf{T}_w\mapsto T_w$ gives an algebra isomorphism $E^G\simeq\HH_n(q).$ \end{prop}
\begin{proof} Since the algebras have the same dimensions, the proposition is proved by checking that the elements $\mathbf{T}_{s_i}$ satisfy the same relations as the $T_i$ in \eqref{relq}. It is a straightforward computation. 


\end{proof}
\begin{rem}The convolution construction is basically the same as defining the algebra of double cosets $B\backslash G/B$, given the identification of the orbit spaces $$G\backslash (G/B\times G/B)\simeq B\backslash G/B.$$ This is completely analogous to the discussion in Section \ref{solom}.\end{rem}
\subsection{Combinatorics and Seminormal representations}\label{seminorm}
A \emph{partition} $\lambda$ of the nonnegative integer $n$ (denoted by $\lambda\vdash n$ or by $|\lambda|=n$) is a nonincreasing sequence of nonnegative integers $\lambda=(\lambda_1,\ldots,\lambda_k)$ such that $\lambda_1\geq\ldots\geq\lambda_k$ and $\lambda_1+\ldots+\lambda_k=n$.
Partitions can be thought of as Young diagrams, consisting of $\lambda_i$ left justified boxes in the $i$-th row from the top.
A Young diagram filled with positive integers is called a \emph{Standard tableau} if all the numbers $1,2,\ldots,d$ are used and they are increasing along the rows and down the columns of the diagram. The partition giving the diagram will be called the \emph{shape} of the tableau.
\begin{exa}\label{stdtbl}A Standard tableau of shape $(4,2,2,1)$ is the following.
$$\young(1358,27,49,6)$$
\end{exa}
For a fixed partition $\lambda$, we denote the set of all standard tableaux of shape $\lambda$ by $\T(\lambda)$.
 \begin{defn}\label{content}Given a tableau $T$, and an integer $1\leq i\leq n$, such that the box containing $i$ is found in the $l$-th row and $m$-th column of $T$, we define the \emph{content} $c_T(i)=m-l$. \end{defn}
\begin{rem}The content identifies which diagonal of the tableau the number $i$ is in. If we let $T$ be the standard tableau of Example \ref{stdtbl}, then we have the sequence of contents
$$ (c_T(1),c_T(2),\ldots, c_T(9))=(0,-1,1,-2,2,-3,0,3,-1)$$
\end{rem}
The Hecke algebra is semisimple for generic $\q$, more precisely when $\q$ is not specialized to a root of unity. When this is the case, the irreducible representations of $\HH_n$ are parametrized by partitions of $n$. One way of constructing them, which was  originally done in \cite{Ho}, is as follows (this exposition is more similar to \cite[Section 3]{Ra}).
Let $\theta\vdash n$, then we have the corresponding irreducible representation 
$$V^\theta:=\CC(\q)\{v_T|T\in\T(\theta)\}$$
with action of the generators given by
\begin{equation*}T_i v_T=\left\{\begin{array}{cl} \q v_T & \text{ if $i,i+1$ are on the same row of $T$} \\
-v_T & \text{ if $i,i+1$ are on the same column of $T$}\\
\frac{\q-1}{1-\q^k}v_T+\left(1+\frac{\q-1}{1-\q^k}\right)v_{s_i(T)} & \text{otherwise}
\end{array}\right.
\end{equation*}
where $s_i(T)$ is the standard tableau obtained from $T$ by exchanging $i$ and $i+1$, and $k=c_T(i)-c_T(i+1)$.
\section{Mirabolic Hecke Algebra}\label{mirha}
We can now introduce the main object of study of this paper. For this section we will use the same notation as in \cite{T}.  As in Section \ref{iwh}, we let $G=\gl_d(\Fq)$, $V=\Fq^d$ and we consider $G$- diagonal orbits on $G/B\times G/B\times V$. These are indexed by pairs $(w,\beta)$ with $w\in S_n$, and $\beta\subset\{1,\ldots,n\}$ satisfying the condition that if $i\in\beta$, $j\not\in\beta$, then $i<j$ or $w(i)<w(j)$. 
\begin{rem}The orbit corresponding to $(w,\beta)$ consists of the triples $(F,F',v)$ such that there exists a basis $\{e_i|i=1,\ldots,d\}$ of $V$ with
$$F_i=  \seq{e_1,\ldots,e_i};\quad F'_j=  \seq{e_{w(1)},\ldots,e_{w(j)}};\quad v=\sum_{i\in\beta}e_i.$$
\end{rem} 

Let $\R_n(q)=\CC(G/B\times G/B\times V)^G\simeq\CC(G\backslash (G/B\times G/B\times V))$ be the space of $G$ invariant functions. Travkin in \cite{T} has studied this space as an $\HH_n(q)$-bimodule. The action is given by convolution: if we realize  $\HH_n(q)$ as the convolution algebra of Proposition \ref{isoconv}, for $\alpha\in\HH_n(q)$ and $\beta\in \R_n(q)$ we can define 
\begin{equation}\label{bimod} \alpha*\beta(F,F',v)=\sum_{H\in\F}\alpha(F,H)\beta(H,F',v);\quad \beta*\alpha(F,F',v)=\sum_{H\in\F}\beta(F,H,v)\alpha(H,F').\end{equation}
It is immediate that this gives well defined $G$-invariant functions.

More interestingly, Solomon in \cite{So} had defined an associative algebra structure on $\R_n(q)$, which can be stated in terms of convolution as follows.
If, $\alpha,\beta\in R_n(q)$,
\begin{equation}\label{mirconv}\alpha*\beta(F,F',v)=\sum_{H\in \F, u\in V}\alpha(F,H,u)\beta(H,F',v-u).\end{equation}
Denoting the characteristic functions of orbits by $T_{w,\beta}$ as in \cite{T}, the identity element for this product is $T_{e,\oo}$ which is the characteristic function of the orbit $\{(F,F',v)|F=F',v=0\}$.
\begin{defn}Since all the structure constants appearing from the product \eqref{mirconv} are polynomials in $q$, we can consider $\R_n(q)$ to be the specialization at $\q\mapsto q$ of a $\mathbb{C}[\q,\q^{-1}]$-algebra $R_n$. We call $R_n$  the \emph{Mirabolic Hecke Algebra}. We will mostly work with the generic mirabolic Hecke algebra which is
$$\R_n:=R_n\otimes_{\CC[\q,\q^{-1}]}\CC(\q)$$\end{defn}
\begin{rem}
The Mirabolic Hecke Algebra contains $\HH_n$ as a subalgebra, with the inclusion being given by $T_w\mapsto T_{w,\oo}$. Given the definition of the products in \eqref{bimod} and \eqref{mirconv}, this inclusion agrees with the bimodule structure defined by Travkin.
\end{rem}
\begin{rem}\label{antiinv}The involution on $\F\times\F\times V$ defined by $(F,F',v)\mapsto(F',F,v)$ induces an algebra anti-automorphism $^\star:\R_n\to\R_n$. In the natural basis for $\R_n$, this can be written as $(T_{w,\beta})^\star= T_{w^{-1},w(\beta)}$.
\end{rem}
\subsection{Comparison with Solomon's Conventions}\label{solom}
In this section we explain how to translate the work from Solomon's paper \cite{So} into the notation that we use.
We consider the group $P$ of affine transformations in $V$, which is isomorphic to the semidirect product $\gl(V)\ltimes V$. We can think of $P$ as the group of $n+1\times n+1$ matrices that fixes a nonzero vector, i.e. the \emph{mirabolic} subgroup of $\gl_{n+1}$. Following the convention of \cite{So}, we can write it as block matrices as follows:
$$P=\left\{\left.\begin{pmatrix} 1 & 0 \\ v & g \end{pmatrix}\right|v\in V,g\in\gl(V)\right\}.$$
Solomon then considers the Borel subgroup $B\subset\gl(V)$ as a subgroup of $P$ and defines the algebra of double cosets $B\backslash P/B$. The multiplication is given by the following formula, if $p_1,\ldots,p_m$ are a set of representatives for the double cosets,
$$ (Bp_iB)\cdot(Bp_jB)=\sum_{k=1}^m c^k_{ij}(Bp_kB)$$
with
\begin{equation}\label{coeff}c^k_{ij}=\df{\card\{(Bp_iB)^{-1}p_k\cap Bp_jB\}}{\card(B)}\end{equation}
which is the same as saying that the coefficient $c^k_{ij}$ equals $\card(B)$ times the number of pairs $(x,y)$, $x\in Bp_iB$, $y\in Bp_jB$ such that $xy=p_k$.
\begin{rem}Another equivalent way of defining the double coset algebra is the following: let $\CC[P]$ be the group algebra of the finite group $P$, and let $e_B:=\frac{1}{\card(B)}\sum_{b\in B}b\in\CC[P]$ be the idempotent corresponding to the  subgroup $B$. Then the algebra of double cosets is isomorphic to $e_B\CC[P]e_B$.
\end{rem}

Solomon has given a description of the double coset algebra 
in terms of generators and relations in \cite{So}. This then gives us a presentation of $R_n$ (and consequently of the generic algebra $\R_n$ and the specialized algebras $\R_n(q)$). 
\begin{thm}\label{solomon}The algebra $R_n$ is isomorphic to the $\CC[\q,\q^{-1}]$-algebra with generators $\{T_i|i=0,1,\ldots, n-1\}$ and relations
\begin{align}T_0^2 & =(\q-2)T_0+(\q-1) & \label{rel1}\\
T_i^2 & = (\q-1)T_i+\q & \quad i\geq 1 \label{rel2}\\
T_iT_{i+1}T_i& =T_{i+1}T_iT_{i+1} & \quad i\geq 1 \label{rel3}\\
T_0T_1T_0T_1& =(\q-1)(T_1T_0T_1+T_1T_0)-T_0T_1T_0 & \label{rel4}\\
T_1T_0T_1T_0& =(\q-1)(T_1T_0T_1+T_0T_1)-T_0T_1T_0 & \label{rel5}\\
T_iT_j&=T_jT_i & \quad |i-j|\geq 2 \label{rel6} 
\end{align}
\end{thm}
The isomorphism is given by $T_0\mapsto T_{e,\{1\}}$, corresponding to the orbit $\{(F,F',v)|F=F',v\in F_1\setminus\{0\}\}$. For $i\geq 1$, $T_i\mapsto T_{s_i,\oo}$, which are the simple reflections that generate $H_n$ as a subalgebra of $R_n$. In fact relations \eqref{rel2}, \eqref{rel3} \eqref{rel6} are the same as \eqref{relq}.
\begin{proof}All that we need to do is to show that the algebra of double cosets is isomorphic to $\R_n(q)$, then the presentation follows from \cite[Thm 6.6]{So}. 
To see why this is true, first of all we need to observe that we have the following isomorphisms of sets of orbits
$$B\backslash P/B\simeq G\backslash(G/B\times P/B)\simeq G\backslash (G/B\times G/B\times P/G)\simeq G\backslash (G/B\times G/B\times V)$$
with the composition being given by
$$B\begin{pmatrix} 1 & 0 \\ v & g \end{pmatrix}B\mapsto G\cdot(B,gB,v).$$
Now, suppose we are given $p_i, p_j,p_k\in P$ and we want to compute $c^k_{ij}$. Write $$p_i=\begin{pmatrix} 1 & 0 \\ v_i & g_i \end{pmatrix}, p_j=\begin{pmatrix} 1 & 0 \\ v_j & g_j \end{pmatrix}, p_k=\begin{pmatrix} 1 & 0 \\ v_k & g_k \end{pmatrix}.$$ 
Counting the pairs of elements $(x,y)$ such that $xy=p_k$ is the same as counting the set $\{(g,g',u,u')|gg'=g_k, gu'+u=v_k\}$, where $g\in Bg_iB$, $g'\in Bg_jB$, $u\in Bv_i$, $u'\in Bv_j$. Let $\Omega_i$ be the orbit $G\cdot(B,g_iB,v_i)$ and similarly for $\Omega_j$ and $\Omega_k$. To show that the structure constants given by the convolution in \eqref{mirconv} are the same as the coefficients in \eqref{coeff}, consider $(F,F',v):=(B,g_kB,v_k)\in\Omega_k$. We will compute the value of the convolution product of characteristic functions $\mathbf{1}_{\Omega_i}*\mathbf{1}_{\Omega_j}(F,F'v)$ and show that it equals $c^k_{ij}$. 
\begin{align*}\mathbf{1}_{\Omega_i}*\mathbf{1}_{\Omega_j}(F,F',v)&=\sum_{H\in G/B,w\in V}\mathbf{1}_{\Omega_i}(F,H,v-w)\mathbf{1}_{\Omega_j}(H,F',w) \\
\mathbf{1}_{\Omega_i}*\mathbf{1}_{\Omega_j}(B,g_kB,v_k)&=\sum_{rB\in G/B,w\in V}\mathbf{1}_{\Omega_i}(B,rB,v_k-w)\mathbf{1}_{\Omega_j}(rB,g_kB,w) \end{align*}
For the first factor on the RHS to be nonzero (and to equal one) we need $rB=g_iB$ (hence $r\in Bg_iB$), and $B(v_k-w)=Bv_i$. For the second factor in the RHS we need $r^{-1}g_kB=g_jB$, (hence $r^{-1}g_k\in Bg_jB$), and $Br^{-1}w=Bv_j$. 

The value of $\mathbf{1}_{\Omega_i}*\mathbf{1}_{\Omega_j}(F,F',v)$ is given by counting all the $r$'s and $w$'s that satisfy those conditions. If we let $g=r,g'=r^{-1}g_k,u=v_k-w,u'=r^{-1}w$, then we see that this is the exact same thing as counting the set $\{(g,g',u,u')|gg'=g_k, gu'+u=v_k\}$ as before and then dividing by $\card(B)$, because we are counting $rB$ as left cosets, therefore it is the same as $c^k_{ij}$.
\end{proof}

\begin{rem}The anti-involution $^\star$ of Remark \ref{antiinv} leaves invariant the generators: $T_i^\star=T_i$ for $i=0,\ldots,n-1$.\end{rem}
\subsection{Irreducible representations of $\R_n$}\label{irrepsrn}
For general values of $q\in\CC^{\times}$, the algebra $\R_n(q)$ is a semisimple algebra of dimension $\sum_{k=0}^n k! \binom{n}{k}^2$, and its irreducible representations have been described by Siegel in \cite{Si}. They are parametrized by pairs $(\theta,k)$, where $0\leq k\leq n$ and $\theta$ is a partition of $k$. In \cite[Prop 3.10]{Si} there is an explicit formula for the action of $\R_n(q)$ on the irreducible left module $M_n^{\theta,k}$. With some similar computations to the ones that we will do in the proof of Proposition \ref{isoiso}, it can be shown to be equivalent to what is described in the following proposition. We state the result for the generic algebra $\R_n$, instead of the specializations, but nothing changes.

Fix $k$ such that $0\leq k\leq n$. Let $\HH_{(k)}\subset\HH_n$ be the subalgebra generated by $T_{s_1},\ldots,T_{s_{k-1}}$, and let $\HH_{(n-k)}\subset\HH_n$ be the subalgebra generated by $T_{s_{k+1}},\ldots,T_{s_{n-1}}$. They are isomorphic to the Iwahori-Hecke algebras of $S_k$ and $S_{n-k}$ respectively. 
\begin{prop}\label{irrep} There is an isomorphisms of left $\HH_n$-modules
\begin{equation}M_n^{\theta,k}\simeq\HH_n\otimes_{\HH_{(k)}\otimes\HH_{(n-k)}} V^\theta\boxtimes \CC_{\text{sign}};\end{equation}
where $V^\theta$ is the irreducible representation of $\HH_{(k)}$ corresponding to the partition $\theta\vdash k$ and $\CC_{\text{sign}}$ is the one dimensional sign representation of $\HH_{(n-k)}$.

Moreover, consider the basis $\{T_w\otimes v_T\}_{w,T}$ of $M_n^{\theta,k}$, where $w\in S_n/(S_k\times S_{n-k})$ such that $w$ is the element with minimal length in the coset, and $\{v_T|T\in\T(\theta)\}$ is a basis of $V^\theta$.
Then 
$$T_0\cdot (T_w\otimes v_T)=\left\{\begin{array}{cc} (\q-1)T_w\otimes v_T & \text{ if }w^{-1}(1)\in\{1,\ldots,k\} \\
-T_w\otimes v_T & \text{ if }w^{-1}(1)\in\{k+1,\ldots,n\} \end{array} \right. .$$
\end{prop}
This explicit description, in particular, tells us that $\dim M_n^{\theta,k}=\binom{n}{k} f_\theta$, where $f_\theta$ is the number of standard tableaux of shape $\theta$. 
Siegel has also described the restriction functor from $\R_n$ to $\R_{n-1}$.
\begin{prop}[\cite{Si}, Cor 3.23]\label{siegelres}
\begin{equation} \Res_{\R_{n-1}}^{\R_n}M_n^{\theta,k}=M_{n-1}^{\theta,k}\oplus\bigoplus_{\nu\vdash k-1}M_{n-1}^{\nu,k-1}\end{equation}
where the term $M_{n-1}^{\theta,k}$ is zero if $k>n-1$ and the sum is over all partitions $\nu$ of $k-1$ that are obtained from $\theta$ by removing one box.
\end{prop}
By Proposition \ref{siegelres} and Frobenius reciprocity we obtain immediately:
\begin{prop}
\begin{equation}\Ind_{\R_{n}}^{\R_{n+1}}M_{n}^{\theta,k}=M_{n+1}^{\theta,k}\oplus\bigoplus_{\mu\vdash k+1}M_{n+1}^{\mu,k+1} \end{equation}
where the sum is over all partitions $\mu$ of $k+1$ that are obtained from $\theta$ by adding one box.
\end{prop}  

\section{The cyclotomic Hecke algebra}\label{cycloha}
\begin{defn}\label{cyclo}Let $u_1,\ldots,u_r\in\CC$, then the \emph{cyclotomic Hecke algebra} $H_n(u_1,\ldots,u_r)$ is the $\CC[\q,\q^{-1}]$-algebra with generators $X,T_1\ldots,T_{n-1}$ and relations
\begin{align*} T_i^2 & = (\q-1)T_i+\q &  i\geq 1\\
 T_iT_{i+1}T_i& =T_{i+1}T_iT_{i+1} &  i\geq 1 \\
 T_iT_j&=T_jT_i & |i-j|\geq 2 \\
XT_i&=T_iX & i\geq 2 \\
XT_1XT_1&=T_1XT_1X & \\
(X-u_1)\cdots(X-u_r)&=0
\end{align*}
\end{defn}
This is an algebra of rank $r^k k!$ over $\CC[\q,\q^{-1}]$, and it was introduced by Ariki and Koike as a deformation of the group algebra of the complex reflection group $G(r,1,k)=(\ZZ/r\ZZ)\wr S_n$. We now recall some known results about this algebra.
\begin{thm}[\cite{Ar}]\label{ariki}If $q\in\CC^\times$, the specialization $$H_n(u_1,\ldots,u_r;q):=H_n(u_1,\ldots,u_r)/(\q-q)$$ is semisimple if and only if
\begin{align*}\quad q^d u_i&\neq u_j \quad & 1\leq i,j\leq r, i\neq j,&\quad -n<d<n;\\
\quad [n]_q!&\neq 0 \quad & \text{where}\quad [n]_q!=\prod_{k=1}^n [k]_q,&\quad [k]_q=1+q+\ldots+q^{k-1}.\end{align*} 
\end{thm}
\begin{defn}As before, we define the generic cyclotomic Hecke algebra to be
$$\HH_n(u_1,\ldots,u_r):= H_n(u_1,\ldots,u_r)\otimes_{\CC[\q,\q^{-1}]}\CC(\q).$$
\end{defn}
Notice in particular that  in the generic algebra $\HH_n(u_1,\ldots,u_r)$ the condition $[n]_\q !\neq 0$ is always satisfied.
\begin{thm}[\cite{AK}, Theorem 3.7]\label{arko}When $H_n(u_1,\ldots,u_r;q)$ is semisimple, its irreducible representations are parametrized by $r$-partitions of $n$. That is, by $\ull=(\lambda^1,\ldots,\lambda^r)$ where $\lambda^1,\ldots,\lambda^r$ are all partitions and $|\lambda^1|+\ldots+|\lambda^r|=n$.
\end{thm}
Ariki and Koike give an explicit construction of each of those representations. We will only discuss a special case of it, in the next section.
\subsection{Representation theory of $\HH_n(1,0)$}\label{repshn}
We will focus on the case of the cyclotomic Hecke algebra $\HH_n(1,0)$. This means that we pick parameters $r=2$, $u_1=1$, and $u_2=0$ in Definition \ref{cyclo}. The defining relation $(X-u_1)\cdots (X-u_r)=0$ then becomes simply $ (X-1)X=0$ or
\begin{equation}\label{xeqe} X^2=X.\end{equation}
In this situation, by Theorem \ref{ariki}, $\HH_n(1,0)$ is semisimple because $[n]_\q !\neq 0$ and $\q^d\cdot 1\neq 0$ for all $d\in\ZZ$. Then, according to Theorem \ref{arko}, the representations of $\HH_n(1,0)$ are parametrized by bipartitions of $n$, $\ull=(\lambda^1,\lambda^2)$, $|\lambda^1|+|\lambda^2|=n$. 
\begin{rem}All the results of this section also hold for all the specializations $H(1,0;q)$ that are semisimple, even if we only state them for $\HH_n(1,0)$.\end{rem}
\begin{defn}Given a bipartition $\ull$, a \emph{standard bi-tableau} of shape $\ull$ is a pair $U=(U^1,U^2)$, with $U^i$ a tableau of shape $\lambda^i$, such that $U$ is filled with the numbers $1,\ldots,n$ in an increasing way from left to right on rows of the same tableau and from top to bottom on columns. \end{defn}
\begin{exa}Consider $\ull=((2,2),(3,1,1))$, a bipartition of $9$, then an example of a standard bi-tableau of shape $\ull$ is
$$ \left(\young(14,38), \young(269,5,7)\right).$$ 
\end{exa}
\begin{defn}Given a standard bi-tableau, $U=(U^1,U^2)$ with $n$ boxes, and an integer $1\leq i\leq n$ such that $i\in U^j$, we define the \emph{content} 
$$c_U(i)=(j-1)+c_{U^j}(i).$$ 
Here $c_{U^j}(i)$ is as in Definition \ref{content}.
\end{defn}
Given a bipartition $\ull$ of $n$, the corresponding irreducible representation $M^{\ull}$ of $\HH_n(1,0)$ is constructed as follows (see \cite{AK}). Let $\T(\ull)$ be the set of standard bi-tableaux of shape $\ull$, then $$M^{\ull}=\CC(\q)\{v_U|U=(U^1,U^2)\in\T(\ull)\}$$ with the action of the generators given by
\begin{equation}\label{action}X v_U=\left\{\begin{array}{cl} v_U & \text{ if }1\in U^1; \\ 0 & \text{ if } 1\in U^2 \end{array}\right. 
\end{equation}
\begin{equation*}T_i v_U=\left\{\begin{array}{cl} \q v_U & \text{ if $i,i+1$ are on the same row} \\
& \quad\text{ of the same tableau;}\\
-v_U & \text{ if $i,i+1$ are on the same column}\\
&\quad\text{ of the same tableau;} \\
\frac{\q-1}{1-\q^k}v_U+\left(1+\frac{\q-1}{1-\q^k}\right)v_{s_i(U)} & \text{ if $i,i+1$ are in the same tableau}\\
&\quad\text{ but not in the same row or column;} \\
v_{s_i(U)} & \text{ if $i\in U^1$, $i+1\in U^2$;} \\
& \\
(\q-1)v_U+\q v_{s_i(U)} & \text{ if $i\in U^2$, $i+1\in U^1$.}\end{array}\right.
\end{equation*}
where $s_i(U)$ is the bi-tableau obtained from $U$ by exchanging $i$ and $i+1$, and $k=c_{U}(i)-c_{U}(i+1)$.

There are obvious inclusions $\HH_{n-1}(1,0)\rinto \HH_n(1,0)$, and restriction and induction of irreducible representation has a nice combinatorial description.
\begin{thm}[see \cite{Ho} or {\cite[Theorem 13.6]{Ar2}}]\label{indres}$$\Res^{H_n(1,0)}_{H_{n-1}(1,0)}M^{\ull}=\bigoplus_{\umu}M^{\umu}$$
where the sum is over all bipartitions $\umu$ of $n-1$ that are obtained from $\ull$ by removing one box;
$$\Ind^{\HH_n(1,0)}_{\HH_{n-1}(1,0)}M^{\unu}=\bigoplus_{\ull}M^{\ull}$$
where the sum is over all bipartitions $\ull$ of $n$ that are obtained from $\unu$ by adding one box.
\end{thm}
These restriction and induction rules can be encoded in a Bratteli diagram, as is done in Figure 1 of \cite{HR}.
\subsection{Irreducible representations of $\HH_2(1,0)$}
There are five bipartitions of the number $2$, corresponding to the irreducible representations of $\HH_2(1,0)$:
$$(2,\oo);\quad (11,\oo);\quad (1,1);\quad (\oo, 2);\quad (\oo,11).$$
The dimensions of these are respectively $1,1,2,1,1$. By \eqref{action}, the generators $X$ and $T_1$ act on each of those according to the following table
\begin{center}
\begin{tabular}{|c|c|c|}\hline
$\ull$ & $X$ & $T_1$ \\ \hline
$(2,\oo)$ & $1$ & $\q$ \\ \hline
$(11,\oo)$ & $1$ & $-1$ \\ \hline
$(1,1)$ & $\begin{pmatrix} 1 & 0 \\ 0 & 0 \end{pmatrix}$ & $\begin{pmatrix} 0 & \q-1 \\ 1 & \q \end{pmatrix}$ \\ \hline
$(\oo,2)$ & $0$ & $\q$ \\ \hline
$(\oo,11)$ & $0$ & $-1$ \\ \hline
\end{tabular}
\end{center}
Given this, it is a quick computation to check the following proposition.
\begin{prop}\label{idemp}The central idempotents of $\HH_2(1,0)$ corresponding to the irreducible representations are the following:
\begin{align*} y_{(2,\oo)}&=\q^{-1}(\q+1)^{-1}(T_1XT_1X+XT_1X)  \\
y_{(11,\oo)}&=(\q+1)^{-1}(T_1XT_1X-\q XT_1X)  \\
y_{(1,1)}&=\q^{-1}(\q X +T_1 XT_1+(\q-1)XT_1X-2T_1XT_1X)  \\
 y_{(\oo,2)}&=(\q+1)^{-1}(1-X+T_1-XT_1-T_1X+XT_1X-T_1XT_1+XT_1XT_1)  \\
y_{(\oo,11)}&=(\q+1)^{-1}(\q-\q X-T_1+XT_1+T_1X-XT_1X-\q^{-1}T_1XT_1+\q^{-1}XT_1XT_1)
\end{align*}
\end{prop}
Let $\ull$ be a bipartition of $2$ and let $I^{\ull}$ be the two sided ideal of $\HH_2(1,0)$ generated by $y_{\ull}$. As a left $\HH_2(1,0)$-module, $I^{\ull}$ is isomorphic to $(M^{\ull})^{\dim M^{\ull}}$.
For any $n\geq 2$ consider the inclusion $\HH_2(1,0)\rinto \HH_n(1,0)$ and the two sided ideal $I_n^{\ull}$ of $\HH_n(1,0)$ generated by $I^{\ull}$. By Theorem \ref{indres}, as left $\HH_n(1,0)$-modules
$$ I_n^{\ull} \simeq \bigoplus_{\ull\subset\umu} \left(\dim M^{\umu}\right)^{\dim M^{\umu}}$$
where the sum is over all bipartitions $\umu$ of $n$ that are obtained by adding $n-2$ boxes to $\ull$.
Since $\HH_n(1,0)$ is semisimple, the following proposition follows immediately.
\begin{prop}\label{cyclrep}$$\HH_n(1,0)/I_n^{\ull}\simeq\bigoplus_{\ull\not\subset\unu}\left(\dim M^{\unu}\right)^{\dim M^{\unu}}$$
as left $\HH_n(1,0)$-modules, where the sum is over all bipartitions $\unu$ that cannot be obtained from $\ull$ by adding boxes. Moreover, as an algebra we have
\begin{equation}\label{bipart}\HH_n(1,0)/I_n^{\ull}\simeq\bigoplus_{\ull\not\subset\unu} \End(M^{\unu})\end{equation}
so that the irreducible representations of $\HH_n(1,0)/I_n^{\ull}$ are parametrized by the set of bipartitions of $n$ that do not contain $\ull$.  
\end{prop}
\subsection{$\R_n$ as a quotient}\label{rnquot}
One of the main results of this paper is that the mirabolic Hecke algebra is a quotient of a cyclotomic Hecke algebra.
\begin{thm}$$\R_n\simeq \HH_n(1,0)/I_n^{(\oo,2)}$$
\end{thm}
\begin{proof}In order to prove this, we need a different presentation of $\R_n$. Let $e=\q^{-1}(T_0+1)$.
\begin{lem}\label{newpres}The algebra $\R_n$ is generated by $e,T_1,\ldots, T_{n-1}$. These generators satisfy the following relations.
\begin{align} \label{nrel1} T_i^2 & = (\q-1)T_i+\q &  i\geq 1\\
\label{nrel2} T_iT_{i+1}T_i& =T_{i+1}T_iT_{i+1} &  i\geq 1 \\
\label{nrel3} T_iT_j&=T_jT_i & |i-j|\geq 2 \\
\label{nrel4}e^2&=e & \\
\label{nrel5}eT_i&=T_ie & i\geq 2 \\
\label{nrel6}eT_1eT_1&=T_1eT_1e & \\
\label{nrel7}eT_1eT_1& =T_1eT_1-eT_1e+T_1e+eT_1+e-T_1-1 &
\end{align}
\end{lem}
\begin{proof}[Proof of Lemma]Since $T_0=\q e-1$, clearly $\R_n$ is generated by $e,T_1,\ldots,T_{n-1}$. The quadratic and braid relations \eqref{nrel1}-\eqref{nrel3} are the same as before.

For \eqref{nrel4}, we use \eqref{rel1}:
\begin{align*} e^2 &=\q^{-2}(T_0+1)^2 \\ 
&= \q^{-2}(T_0^2+2T_0+1) \\
&=\q^{-2}((\q-2)T_0+\q-1)+2T_0+1)\\
&=\q^{-2}\q(T_0+1) \\
&= \q^{-1}(T_0+1)=e. \end{align*}
The commutation relation \eqref{nrel5} follows from \eqref{rel6}, since $T_0$ commutes with $T_i$, $i\geq 2$.

To show relations \eqref{nrel7}, we use \eqref{rel4}.
\begin{align*}eT_1eT_1 &=\q^{-2}(T_0+1)T_1(T_0+1)T_1 \\
&=\q^{-2}(T_0T_1T_0T_1+T_1T_0T_1+T_0T_1^2+T_1^2) \\
&=\q^{-2}(((\q-1)T_1T_0T_1+(\q-1)T_1T_0-T_0T_1T_0)+ \\
&\quad +T_1T_0T_1+T_0((\q-1)T_1+\q)+((\q-1)T_1+\q) ) \\
&=\q^{-2}(\q T_1T_0T_1-T_0T_1T_0+(\q-1)T_1T_0+(\q-1)T_0T_1+ \\
&\quad +\q T_0+(\q-1)T_1+\q) \\
&=\q^{-2}(\q T_1(\q e-1)T_1-(\q e-1)T_1(\q e-1)+(\q-1)T_1(\q e-1) \\
&\quad +(\q-1)(\q e-1)T_1+ \q(\q e-1)+(\q-1)T_1+\q) \\
&=\q^{-2}(\q^2T_1eT_1-\q T_1^2-\q^2eT_1e+\q T_1e+\q eT_1-1+\q(\q-1)T_1e-(\q-1)T_1 \\
&\quad + \q (\q -1)eT_1-(\q -1)T_1+\q^2e-\q +(\q-1)T_1+\q )\\
&=T_1eT_1-eT_1e+T_1e+eT_1+e-T_1-1 \end{align*}
Finally, we could do an analogous computation for \eqref{nrel6} or simply remark that, since $e^\star=e$, applying the anti-involution to both sides of \eqref{nrel7} we get
\begin{align*}T_1eT_1e &= (eT_1eT_1)^\star \\
&=( T_1eT_1-eT_1e+T_1e+eT_1+e-T_1-1 )^\star \\
&= T_1eT_1-eT_1e+eT_1+T_1e+e-T_1-1  \\
&= eT_1eT_1. \end{align*}
\end{proof}

\begin{lem}The relations in Lemma \ref{newpres} are equivalent to the relations in Theorem \ref{solomon}, hence that is a new presentation for $\R_n$.\end{lem}
\begin{proof}[Proof of Lemma]We need to show that relations \eqref{rel1} and \eqref{rel4}-\eqref{rel6} follow from relations \eqref{nrel1}-\eqref{nrel7}, if we let $T_0=\q e -1$. The last relation \eqref{rel6} is clear.
For \eqref{rel1}, we have 
\begin{align*}T_0^2&=(\q e -1)^2\\
&= \q^2 e^2 -2\q e+1 \\
&= \q^2 e -2\q e +1 \\
&= \q(\q -2)(\q^{-1}(T_0+1))+1 \\
&= (\q-2)T_0+\q-2+1\\
&=(\q-2)T_0+\q-1.
\end{align*}
Now, we check relation \eqref{rel4}.
\begin{align*}T_0T_1T_0T_1&=(\q e -1)T_1(\q e-1)T_1 \\
&= \q^2eT_1eT_1-\q T_1eT_1-\q e T_1^2+T_1^2 \\
\text{by \eqref{nrel7}}&=\q^2(T_1eT_1-eT_1e+T_1e+eT_1+e-T_1-1)-\q T_1eT_1+ \\
&\quad -\q(\q-1) eT_1-\q^2 e+(\q-1)T_1+\q \\
&= \q(\q-1)T_1eT_1-\q^2eT_1e+\q^2T_1e+\q eT_1+(\q-1-\q^2)T_1+\q-\q^2\\
&=(\q-1)T_1(T_0+1)T_1-(T_0+1)T_1(T_0+1)+\q T_1(T_0+1)+(T_0+1)T_1 \\
&\quad +(\q-1-\q^2)T_1+\q-\q^2  \\
&= (\q-1)T_1T_0T_1+(\q-1)T_1^2-T_0T_1T_0-T_1T_0-T_0T_1-T_1+ \\
&\quad+\q T_1T_0+\q T_1+T_0T_1+T_1+(\q-1-\q^2)T_1+\q-\q^2 \\
&= (\q-1)T_1T_0T_1+(\q-1)T_1T_0-T_0T_1T_0.
\end{align*}
To show \eqref{rel5} we can either do a similar computation to the one we just did or we can apply the anti-involution $^\star$ to both sides of \eqref{rel4} to obtain
\begin{align*}T_1T_0T1T_0&=(T_0T_1T_0T_1)^\star \\
&=(\q-1)(T_1T_0T_1+T_1T_0)^\star-(T_0T_1T_0)^\star \\
&=(\q-1)(T_1T_0T_1+T_0T_1)-T_0T_1T_0.
\end{align*}
\end{proof}
Now, to conclude the proof of the theorem, we compare the relations in Lemma \ref{newpres} with the relations in Definition \ref{cyclo}, taking $e=X$. Clearly \eqref{nrel1}-\eqref{nrel3}, \eqref{nrel5} and \eqref{nrel6} are the same in both. Since we are specialized to $H_n(1,0)$, \eqref{nrel4} is the same as \eqref{xeqe}. All that is left is \eqref{nrel7}, which is equivalent to
$$eT_1eT_1-T_1eT_1+eT_1e-T_1e-eT_1-e+T_1+1=0$$
which, from \eqref{idemp} is the same as
$$y_{(\oo,2)}=0$$
and the result follows.
\end{proof}
We can now recover the results of Siegel on the representation theory of $\R_n$ as a consequence.
First of all, notice that the bipartitions of $n$ parametrizing the irreducible representations of  $\HH_n(1,0)/I_n^{(\oo,2)}$ are, as seen in \eqref{bipart}, the ones that do not contain $(\oo,2)$. If we denote by $a^b$ the partition $(aa\ldots a)$ where $a$ appears $b$ times, then these are exactly the ones of the form $\umu=(\mu^1, 1^{n-k})$, with $\mu^1\vdash k$. As Young diagrams, this means that the second diagram is a single column. There is an obvious bijection between this set and the set $\{\theta\vdash k|0\leq k\leq n\}$ that we used in Section \ref{irrepsrn}, given by
$$ \theta\vdash k\longleftrightarrow (\theta,1^{n-k}).$$
\begin{prop}\label{isoiso}There is an isomorphism of irreducible modules over $\R_n\simeq \HH_n(1,0)/I_n^{(\oo,2)}$
$$ M^{\theta,k}_n\simeq M^{(\theta,1^{n-k})}$$
given, on the basis $\{T_w\otimes v_T\}$ where $w\in S_n/(S_k\times S_{n-k})$ such that $w$ is the element with minimal length in the coset, by
\begin{equation}\label{modiso}T_w\otimes v_T\mapsto v_{(U^1,U^2)}.\end{equation}
Here $U^1=w(T)$ is the standard tableau of shape $\theta$ whose boxes contain the numbers $\{w(1),\ldots, w(k)\}$ arranged so that $w(l)$ is in the same spot as $l$ is in $T$. Then $U^2$ is the standard tableau consisting of a single column with the numbers $w(k+1),\ldots,w(n)$. (Remark that $(U^1,U^2)$ as defined is standard because $w$ is the shortest element in the coset, which means that $w(1)<\ldots <w(k)$ and $w(k+1)<\ldots<w(n)$.)
\end{prop}
\begin{exa}If $w=(1234)=s_1s_2s_3\in S_5$ and $T=\young(13,2)$, then 
$$U=(U^1,U^2)=\left(\young(24,3),\young(1,5)\right).$$
\end{exa}
\begin{proof}[Proof of Proposition \ref{isoiso}]It is clear that the map \eqref{modiso} sends a basis of $M^{\theta,k}_n$ to a basis of $M^{(\theta,1^{n-k})}$, so it is injective and surjective. We need to show that it intertwines the action of $\R_n$. Since $e=\q^{-1}(T_0+1)$, we have, by Proposition \ref{irrep}, that
$$e\cdot(T_w\otimes v_T)=\left\{\begin{array}{cl} T_w\otimes v_T & \text{ if }w^{-1}(1)\in\{1,\ldots,k\} \\
0 & \text{ if }w^{-1}(1)\in\{k+1,\ldots,n\}\end{array}\right. $$
Under the map \eqref{modiso}, $w^{-1}(1)\in\{1,\ldots,k\}$ if and only if $1\in U^1$, and $w^{-1}(1)\in\{k+1,\ldots,n\}$ if and only if $1\in U^2$, so the action of $e$ is the same as the action of $X$ in \eqref{action}.

To conclude we need to check the action of the $T_i$'s. There will be three separate cases. 

Case 1: suppose that $w^{-1}(i)$ and $w^{-1}(i+1)$ are both in $\{1,\ldots,k\}$ or both in $\{k+1,\ldots,n\}$. Then since $w$ is the shortest element in the coset, we have that $l(s_iw)>l(w)$, and $s_iw=ws_{w^{-1}(i)}$, hence
\begin{align*}T_i\cdot(T_w\otimes v_T)&= T_iT_w\otimes v_T \\
&= T_{s_{i}w}\otimes v_T \\
&= T_{ws_{w^{-1}(i)}} \otimes v_T \\
&= T_w \otimes T_{s_{w^{-1}(i)}}\cdot v_T
\end{align*}
Interpreting the action $T_{s_{w^{-1}(i)}}\cdot v_T$ as the seminormal representation of $\HH_n$ as in Section \ref{seminorm}, we see that after applying the isomorphism \eqref{modiso}, this agrees with the cases in \eqref{action} where $i,i+1$ are in the same tableau.

Case 2: $w^{-1}(i)\in\{1,\ldots,k\}$, $w^{-1}(i+1)\in\{k+1,\ldots,n\}$. In this case we have that $l(s_iw)>l(w)$ and $s_iw$ is again the shortest element in its coset, therefore
\begin{align*} T_i\cdot(T_w\otimes v_T)&= T_iT_w\otimes v_T \\
&= T_{s_{i}w}\otimes v_T \end{align*}
and by \eqref{modiso} this is mapped to $v_{(s_iw(T),U^2)}=v_{s_i(U^1,U^2)}$ which agrees with \eqref{action} when $i\in U^1$ and $i+1\in U^2$.

Case 3: $w^{-1}(i+1)\in\{1,\ldots,k\}$, $w^{-1}(i)\in\{k+1,\ldots,n\}$. In this case we have that $l(s_iw)<l(w)$ and $s_iw$ is again the shortest element in its coset, so
\begin{align*} T_i\cdot(T_w\otimes v_T)&= T_iT_w\otimes v_T \\
&=\left((\q-1)T_w+\q T_{s_{i}w}\right)\otimes v_T \\
& =(\q-1)T_w\otimes v_T+ \q  T_{s_{i}w}\otimes v_T \end{align*}
which maps to $(\q-1)v_U+\q v_{s_i(U)}$ and agrees with \eqref{action} when $i\in U^2$ and $i+1\in U^1$.
\end{proof}
\begin{rem}From Proposition \ref{isoiso} and Theorem \ref{indres} we also immediately recover the rule for restriction of Proposition \ref{siegelres}.\end{rem}
\section{Comparison with the $q$-Rook algebra}\label{qrook}
The $q$-Rook algebra was first introduced by Solomon in \cite{So2} as a double coset algebra. Considering the monoid $M_n(\Fq)$ of $n\times n$ matrices over the finite field with $q$ elements and the Borel subgroup $B$ of upper triangular invertible matrices, we get the $q$-Rook algebra $\II_n(q)$ as the algebra of double cosets $B\backslash M_n(\Fq)/B$. For generic values of the parameter $q$, it is a semisimple algebra and its irreducible representations were described combinatorially by Halverson in \cite{Ha}. Halverson and Ram then showed in \cite{HR} that $\II_n(q)$ is a quotient of $\HH_n(1,0;q)$ and that the description of the irreducible representations of \cite{Ha} followed from the representation theory of $\HH_n(1,0;q)$ analogously to what we just saw in Section \ref{rnquot}. We will summarize those results in the following statement, keeping the notation of Section \ref{rnquot}.
\begin{thm}[See Theorem 1.10 and Corollary 2.21 in \cite{HR}]
$$\II_n(q)\simeq \HH_n(1,0;q)/I_n^{(\oo,11)}.$$
When $[n]_q!\neq 0$, then $\II_n(q)$ is a semisimple algebra and its irreducible representations are parametrized by the set of bipartitions that do not contain $(\oo,11)$. This is the set
$$\{(\lambda^1,\lambda^2)|\lambda^2 \text{ is a single row}\}.$$
The representations are given explicitly by the formulas \eqref{action}.
\end{thm}
To be consistent with our previous notation, let us denote by $\II_n$ the generic version of the $q$-Rook algebra over the field $\CC(\q)$.
\begin{prop}There is an algebra isomorphism $$\II_n\simeq \R_n$$
and similarly for the semisimple specializations (when $[n]_q!\neq 0$)
$$\II_n(q)\simeq \R_n(q).$$
\end{prop}
\begin{proof}A finite dimensional semisimple algebra is just a direct sum of matrix algebras, so two semisimple algebras that have the same number of irreducible representations of the same dimensions are isomorphic. 

Notice that the irreducible representation $M^{(\lambda^1,\lambda^2)}$ of $\HH_n(1,0)$ has dimension equal to the number of standard bi-tableaux of shape $(\lambda^1,\lambda^2)$, which is
$$d(\lambda^1,\lambda^2)={n \choose |\lambda^1|}f_{\lambda^1}f_{\lambda^2}$$
where $f_{\lambda^{j}}$ is the number of standard tableaux of shape $\lambda^j$. If we denote the transposed partition of $\lambda$ by $\lambda^t$, clearly $f_\lambda=f_{\lambda^t}$ because the transpose of a standard tableau is again a standard tableau (of the transposed shape).
It follows immediately that
\begin{equation}\label{symd}d((\lambda_1)^t,(\lambda^2)^t)=d((\lambda^1)^t,\lambda^2)=d(\lambda^1,(\lambda^2)^t)=d(\lambda^1,\lambda^2).\end{equation}
The irreducible representations of $\R_n$ are parametrized by the bipartitions 
$$\{(\lambda^1,\lambda^2)|\lambda^2\text{ is a single column }\}.$$ For $\II_n$, we have the bipartitions $\{(\lambda^1,\lambda^2)|\lambda^2\text{ is a single row}\}$.
Clearly the map
$$(\lambda^1,\lambda^2)\mapsto (\lambda^1,(\lambda^2)^t)$$
gives a bijection between the irreducible representations of $\R_n$ and the ones of $\II_n$ which, by \eqref{symd}, preserves dimensions.
\end{proof}
\begin{rem}In fact the same is true for all four of the quotients of $\HH_n(1,0)$ by the ideal generated by the idempotent corresponding to a one dimensional irreducible representation of $\HH_2(1,0)$. In fact
\begin{align*}\{\operatorname{Irreps} \text{ of }\HH_n(1,0)/I_n^{(2,\oo)}\}&\longleftrightarrow \{(\lambda^1,\lambda^2)|\lambda^1\text{ is a single column}\}\\
\{\operatorname{Irreps} \text{ of }\HH_n(1,0)/I_n^{(11,\oo)}\}&\longleftrightarrow \{(\lambda^1,\lambda^2)|\lambda^1\text{ is a single row}\}.\end{align*}
Hence
$$(\lambda^1,\lambda^2)\mapsto ((\lambda^1)^t,\lambda^2)$$
gives a bijection between those two sets of irreducible representations that preserves dimensions. Exchanging the role of $\lambda^1$ and $\lambda^2$ we get a bijection between these representations and the ones of $\R_n$ and $\II_n$, which again preserves dimensions because $d(\lambda^1,\lambda^2)=d(\lambda^2,\lambda^1)$.
In conclusion
$$ \R_n\simeq\II_n\simeq\HH_n(1,0)/I_n^{(2,\oo)}\simeq\HH_n(1,0)/I_n^{(11,\oo)}.$$
\end{rem}
The existence of these isomorphisms does not mean that it is easy to write them down explicitly in terms of generators, except for two of them which come from an automorphism of $\HH_n(1,0)$. 
\begin{defn}Define the $\CC(\q)$-linear map $\alpha$ on the generators of $\HH_n(1,0)$ by
\begin{align*}\alpha(X)& = 1-X\\
\alpha(T_i)&= \q-1-T_i = -\q T_i^{-1}\end{align*}
\end{defn}
\begin{rem}The map $\alpha$ is an involution, because $\alpha^2=\operatorname{Id}$.
\end{rem}
\begin{prop}The map $\alpha$ is an algebra automorphism of $\HH_n(1,0)$.
This isomorphism acts on the idempotents of $\HH_2(1,0)$ in the following way:
$$ \alpha(u_{(\oo,2)})=u_{(11,\oo)};\quad \alpha(u_{(\oo,11)})=u_{(2,\oo)}.$$
\end{prop}
\begin{proof}The quadratic relations for $\alpha(X)$ and $\alpha(T_i)$ are easily checked:
\begin{align*}\alpha(X)^2&=(1-X)^2\\
&=1+X^2-2X \\
&=1-X=\alpha(X) \\
\alpha(T_i)^2&=(\q-1-T_i)^2\\
&=(\q-1)^2+T_i^2-2(\q-1)T_i \\
&=(\q-1)^2+(\q-1)T_i+\q-2(\q-1)T_i\\
&=(\q-1)((\q-1)-T_i)+\q\\
&=(\q-1)\alpha(T_i)+\q\end{align*}
Since both $\alpha(X)$ and $\alpha(T_i)$ are a linear combination of the generator and a constant, the image under $\alpha$ of commuting generators still commutes. For the braid relations we have that
\begin{align*} T_iT_{i+1}T_i&= T_{i+1}T_iT_{i+1}\\
\iff  T_i^{-1}T_{i+1}^{-1}T_i^{-1}&= T_{i+1}^{-1}T_{i}^{-1}T_{i+1}^{-1} \\
\iff -\q^3  T_i^{-1}T_{i+1}^{-1}T_i^{-1}&=-\q^3 T_{i+1}^{-1}T_{i}^{-1}T_{i+1}^{-1}\\
\iff \alpha(T_i)\alpha(T_{i+1})\alpha(T_i)&= \alpha(T_{i+1})\alpha(T_i)\alpha(T_{i+1})\end{align*}
The only relation left to check is that
$$ \alpha(X)\alpha(T_1)\alpha(X)\alpha(T_1)=\alpha(T_1)\alpha(X)\alpha(T_1)\alpha(X)$$
which follows from a somewhat tedious computation, using the fact that $XT_1XT_1=T_1XT_1X$.
This proves that $\alpha$ is an algebra automorphism. 
The statement about the action on the idempotents is again proved with a direct computation.
\end{proof}
\begin{cor}The map $\alpha$ induces the isomorphisms
$$\R_n\simeq \HH_n(1,0)/I_n^{(11,\oo)};\qquad \II_n\simeq \HH_n(1,0)/I_n^{(2,\oo)}.$$
\end{cor}
\section{Jucys-Murphy elements, the center of $\R_n$ and categorification}\label{jme}
The Jucys-Murphy (JM) elements in $\CC[S_n]$ were defined in  \cite{J} and \cite{Mu}.  They generate the maximal commutative Gelfand-Zeitlin subalgebra of  $\CC[S_n]$ and act diagonally on the seminormal representations. This point of view has been used by Okounkov and Vershik (\cite{OV}) to approach the study of the representations of $S_n$  in a Lie theoretic way, considering the weights of this commutative subalgebra. Dipper and James (\cite{DJ}) gave $\q$-analogs of the JM elements in $\HH_n$, which also act diagonally on seminormal representations and generate the Gelfand-Zeitlin subalgebra.
\begin{defn}We define the JM elements in $\HH_n(1,0)$ to be, for $i=1,\ldots,n$
\begin{equation*}\tilde{L}_i=\q^{1-i}T_{i-1}\cdots T_1XT_1\cdots T_{i-1}.
\end{equation*}
and we then define the JM elements in $\R_n$ to be the image of these in the quotient, that is
\begin{equation*}L_i=\q^{1-i}T_{i-1}\cdots T_1eT_1\cdots T_{i-1}.
\end{equation*}
\end{defn}
Several properties of the elements $\tilde{L}_i$ are proved in \cite[Lemma 13.2]{Ar2} (the notation is different there, in Ariki's notation we have $\tilde{L}_i=t_i$ and $T_i=a_{i+1}$). In particular we have that 
$$\tilde{L}_i\tilde{L}_j=\tilde{L}_j\tilde{L}_i$$
and if $p(\tilde{L}_1,\ldots,\tilde{L}_n)$ is a polynomial which is symmetric in $\tilde{L}_i$ and $\tilde{L}_{i+1}$, then
\begin{equation}\label{commute}T_i p(\tilde{L}_1,\ldots,\tilde{L}_n)=p(\tilde{L}_1,\ldots,\tilde{L}_n) T_i
\end{equation}
This implies that symmetric polynomials in the $\tilde{L}_i$'s are in the center of $\HH_n(1,0)$.
\begin{prop}[{\cite[Corollary 13.8]{Ar2}}]\label{diag}Let $U=(U^1,U^2)\in\T(\ull)$ be a bi-tableau and $v_U$ the corresponding basis element in $M^{\ull}$. Then
$$ \tilde{L}_i v_U=\ee_i \q^{c_U(i)}v_U;\quad\text{ with }\quad
\ee_i=\left\{\begin{array}{cl} 1 & \text{ if }i\in U^1 \\ 0 & \text{ if }i\in U^2\end{array}\right. ;$$
\end{prop}
\begin{cor}\label{corol}
If $\ull=(\lambda^1,1^m)$ so that $M^{\ull}$ is a representation of $\R_n$ then, with the same conventions as in Proposition \ref{diag}, we have
$$ L_i v_U=\ee_i \q^{c_U(i)}v_U.$$
\end{cor}
We can describe the center of $\R_n$ in terms of symmetric polynomials in the JM elements.
\begin{thm}$$ Z(\R_n)=\CC(\q)[L_1,\ldots,L_n]^{S_n}$$
and analogously for the semisimple specializations $\R_n(q)$.
\end{thm}
\begin{proof}Since symmetric polynomials in the $\tilde{L}_i$ are in the center of $\HH_n(1,0)$ and the surjection
$$\HH_n(1,0)\ronto \R_n$$
gives $\tilde{L}_i\mapsto L_i$, we have the easy inclusion 
$$Z(\R_n)\supset\CC(\q)[L_1,\ldots,L_n]^{S_n}.$$
For the other inclusion, since $\R_n$ is semisimple, all we need to show is that the action of $\CC(\q)[L_1,\ldots,L_n]^{S_n}$ distinguishes the irreducible representations of $\R_n$. Let $\ull=(\lambda^1,\lambda^2)$ be a bipartition with $\lambda^2$ a single column, and consider the irreducible representation $M^{\ull}$ of $\R_n$. For a bitableau $U\in\T(\ull)$, we have, by Corollary \ref{corol}, that the JM elements $(L_1,\ldots,L_n)$ act on $v_U$ by the eigenvalues $(\ee_1\q^{c_U(1)},\ldots,\ee_n\q^{c_U(n)})$. It follows that the elementary symmetric polynomial $e_i(L_1,\ldots,L_n)$ acts on all the basis elements $\{v_U|U\in\T(\ull)\}$ of $M^{\ull}$ by the same eigenvalue 
$$a_i=e_i(\ee_1\q^{c_U(1)},\ldots,\ee_n\q^{c_U(n)})$$
which is independent of $U$. We claim that the values $a_1,\ldots,a_n$ determine $M^{\ull}$ uniquely, which will imply the theorem.

Given $a_1,\ldots,a_n$, by definition of the elementary symmetric functions, we can recover the unordered multiset $S=\{\ee_1\q^{c_U(1)},\ldots,\ee_n\q^{c_U(n)}\}$ from the equation
$$(t-\ee_1\q^{c_U(1)})\cdots(t-\ee_n\q^{c_U(n)})=t^n-a_1 t^{n-1}+\ldots +(-1)^{n-1}a_{n-1}t+(-1)^na_n.$$
We then obtain $(\lambda^1,\lambda^2)$ in the following way: $\lambda^2$ is the partition with a single column with number of boxes equal to the multiplicity of zero appearing in $S$. For each $j\in\ZZ$, if we denote by $d_j$ the number of times the number $\q^j$ appears in $S$, then $\lambda^1$ is the unique partition with $d_j$ boxes on the $j$-th diagonal.
\end{proof}
\begin{rem}This is completely analogous to the proof of \cite[Prop 3.22]{RR}, but it should be noticed that in that proof it was necessary that none of the parameters $u_1,\ldots,u_r$ of the cyclotomic Hecke algebra be zero. In fact $\CC(\q)[\tilde{L}_1,\ldots,\tilde{L}_n]^{S_n}\subsetneq Z(\HH_n(1,0))$, since the $\tilde{L}_i$'s cannot distinguish the representations $M^{(\lambda^1,\lambda^2)}$ and $M^{(\lambda^1,\lambda^3)}$.
\end{rem}
\subsection{Refined restriction and induction functors}
Consider the JM element $L_n\in\R_n$. For $1\leq i\leq n-2$, we have that $T_iL_n=L_nT_i$, and, since $L_1=e$, we also have $e L_n=L_n e$. It follows that the image of the inclusion
$$\R_{n-1}\rinto \R_n$$
commutes with $L_n$. We denote by $\R_n$-$\Mod$ the category of finitely generated left $\R_n$-modules. Given  $M\in\R_n$-$\Mod$, it has a vector space decomposition into the direct sum of the generalized eigenspaces for $L_n$ and since $\R_{n-1}$ commutes with the $L_n$ action, each of those generalized eigenspace is a module for $\R_{n-1}$.
\begin{defn}For $M\in\R_n$-$\Mod$ and an eigenvalue $\alpha$ of $L_n$, we denote by $M(\alpha)\in\R_{n-1}$-$\Mod$ the generalized eigenspace with eigenvalue $\alpha$.
\end{defn}
Since $\R_n$ is semisimple, each module $M$ decomposes as a direct sum of $M^{\ull}$. By Corollary \ref{corol}, since $L_n$ acts diagonally on the basis $v_U$, we get that $M(\alpha)$ is an actual eigenspace and that the possible eigenvalues for $L_n$ are in the set $\{\q^i|i\in\ZZ\}\cup\{0\}$.
\begin{defn}For $i\in\ZZ$, we define the refined restriction functor 
$$i\text{-}\Res^n_{n-1}:\R_n\text{-}\Mod\to\R_{n-1}\text{-}\Mod\quad\text{ by }\quad M\mapsto M(\q^i);$$
we also define the functor 
$$\infty\text{-}\Res^n_{n-1}:\R_n\text{-}\Mod\to\R_{n-1}\text{-}\Mod\quad\text{ by }\quad M\mapsto M(0).$$
\end{defn}
\begin{prop}\label{ires}Let $(\lambda^1,1^k)$ be a bipartition of $n$, so that $M^{(\lambda^1,1^k)}$ is an irreducible $\R_n$-module. \begin{itemize}
\item[(i)] For $i\in\ZZ$, we have that 
$$i\text{-}\Res^n_{n-1}M^{(\lambda^1,1^k)}=M^{(\lambda^1[i^-],1^k)}$$
where $\lambda^1[i^-]$ is the partition obtained from $\lambda^1$ by removing a box of content $i$ (there is at most one box of content $i$ that can be removed to give another partition). If there is no such box, then $M^{(\lambda^1[i^-],1^k)}=0$.
\item[(ii)] We also have 
$$\infty\text{-}\Res^n_{n-1}M^{(\lambda^1,1^k)}=M^{(\lambda^1,1^{k-1})}.$$
when $k\geq 1$. If $k=0$ then $M^{(\lambda^1,1^{k-1})}=0$.
\end{itemize}
\end{prop}
\begin{proof}
From Corollary \ref{corol}, the eigenspace $M^{(\lambda^1,1^k)}(\q^i)$ is spanned by the $v_{(U^1,U^2)}$ such that $n\in U^1$ and $c_U(n)=i$. If there are no such $U$'s then the restriction is the zero module. Otherwise, since the action of $e,T_1,\ldots,T_{n-2}$ doesn't affect the position of the box containing $n$ in the tableau, we have a map
\begin{equation}\label{is} i\text{-}\Res^n_{n-1}M^{(\lambda^1,1^k)}\to M^{(\lambda^1[i^-],1^{k})}\end{equation}
$$v_{(U^1,1^k)}\mapsto v_{(U^1\setminus\{n\},1^k)}$$
where $U^1\setminus\{n\}$ is the tableau obtained by removing the box containing $n$. This is an isomorphism by the definition of the action in \eqref{action} and it proves (i).

For (ii) it is completely analogous: $M^{(\lambda^1,1^k)}(0)$ is spanned by $v_{(U^1,U^2)}$ such that $n\in U^2$, and in this case $n$ will be in the box at the bottom of the single column of $U^2$. We conclude that we have the analogous isomorphism to \eqref{is} by removing the box containing $n$.
\end{proof}
Since both taking generalized eigenspaces and restriction are exact functors, we have that for all $i\in\ZZ\cup\{\infty\}$, the functor $i$-$\Res^n_{n-1}$ is exact. We then have a left adjoint functor 
$$ i\text{-}\Ind^n_{n-1}:\R_{n-1}\text{-}\Mod\to\R_n\text{-}\Mod $$
and a right adjoint functor
$$ i\text{-}\widehat{\Ind}^n_{n-1}:\R_{n-1}\text{-}\Mod\to\R_n\text{-}\Mod.$$
\begin{prop}\label{iind}Let $(\lambda^1,1^k)$ be a bipartition of $n-1$, so that $M^{(\lambda^1,1^k)}$ is an irreducible $\R_{n-1}$-module.
\begin{itemize}
\item[(i)]For $i\in\ZZ$ we have
$$i\text{-}\Ind^n_{n-1}M^{(\lambda^1,1^k)}=i\text{-}\widehat{\Ind}^n_{n-1}M^{(\lambda^1,1^k)}=
M^{(\lambda^1[i^+],1^k)}$$
where $\lambda^1[i^+]$ is the partition obtained from $\lambda^1$ by adding a box of content $i$ (there is at most one box of content $i$ that can be added to give another partition). If there is no such box, then $M^{(\lambda^1[i^+],1^k)}=0$.
\item[(ii)]We also have 
$$\infty\text{-}\Ind^n_{n-1}M^{(\lambda^1,1^k)}=\infty\text{-}\widehat{\Ind}^n_{n-1}M^{(\lambda^1,1^k)}=
M^{(\lambda^1,1^{k+1})}.$$
\end{itemize}
\end{prop}
\begin{proof}To prove (i), let $i\in\ZZ$. For any $(\mu,1^l)$ bipartition of $n$ consider the irreducible $M^{(\mu,1^l)}$, then we have
\begin{align*}\Hom_{\R_n}(i\text{-}\Ind^n_{n-1}M^{(\lambda^1,1^k)},M^{(\mu,1^l)})&=
\Hom_{\R_{n-1}}(M^{(\lambda^1,1^k)},i\text{-}\Res^n_{n-1}M^{(\mu,1^l)}) \\
&= \Hom_{\R_{n-1}}(M^{(\lambda^1,1^k)},M^{(\mu[i^-],1^l)}). \end{align*}
This equals zero unless $k=l$ and $\lambda^1=\mu[i^-]$, in which case it is a one dimensional space. Since  $\lambda^1=\mu[i^-]$ if and only if $\lambda^1[i^+]=\mu$ and $\R_{n-1}$ is semisimple, it follows that $$i\text{-}\Ind^n_{n-1}M^{(\lambda^1,1^k)}=M^{(\lambda^1[i^+],1^k)}.$$
In the same way
\begin{align*}\Hom_{\R_n}(M^{(\mu,1^l)},i\text{-}\widehat\Ind^n_{n-1}M^{(\lambda^1,1^k)})&=
\Hom_{\R_{n-1}}(i\text{-}\Res^n_{n-1}M^{(\mu,1^l)},M^{(\lambda^1,1^k)}) \\
&= \Hom_{\R_{n-1}}(M^{(\mu[i^-],1^l)},M^{(\lambda^1,1^k)}). \end{align*}
which equals zero unless $k=l$ and $\lambda^1=\mu[i^-]$, in which case it is a one dimensional space. Hence, as above $$i\text{-}\Ind^n_{n-1}M^{(\lambda^1,1^k)}=M^{(\lambda^1[i^+],1^k)}.$$
The exact same argument works for $\infty\text{-}\Ind^n_{n-1}$ and $\infty\text{-}\widehat{Ind}^n_{n-1}$ to prove (ii).
\end{proof}
\begin{cor}For all $i\in\ZZ\cup\{\infty\}$, we have an isomorphims of functors 
$$i\text{-}\Ind^n_{n-1}\simeq i\text{-}\widehat{\Ind}^n_{n-1}.$$ It follows that  $i$-$\Ind^n_{n-1}$ is an exact functor.
\end{cor}
\begin{proof}Since $\R_{n-1}$ is semisimple, every module decomposes as a direct sum of irreducibles. The two functors agree on the irreducibles, therefore are isomorphic. Since $i$-$\Ind^n_{n-1}$ is both left and right adjoint to $i$-$\Res^n_{n-1}$, it is an exact functor.
\end{proof}
\subsection{Categorification}
We denote by $G_0(\R_n)=\K_0(\R_n$-$\Mod)$ the Grothendieck group of the abelian category $\R_n$-$\Mod$, and for $M\in\R_n$-$\Mod$, we denote by $[M]$ its class in $G_0(\R_n)$. Remark that given our description of the simple modules of $\R_n$, we have an isomorphism
$$G_0(\R_n)\simeq \bigoplus_{\ull}\ZZ[M^{\ull}]$$
where the sum is taken over $\ull=(\lambda^1,\lambda^2)$ with $\lambda^2$ a single column. 
For simplicity, we will work over the complex numbers and we will let
$$ \G(\R_n)=\CC\otimes_{\ZZ} G_0(\R_n)\simeq \bigoplus_{\ull}\CC[M^{\ull}]$$
with $\ull$ as above.
For $i\in\ZZ\cup\{\infty\}$, we have the exact functors $i\text{-}\Res^n_{n-1}M$ and $i\text{-}\Ind^n_{n-1}M$ which induce  linear maps
$$ e_{i,n}:\G(\R_n)\to \G(\R_{n-1}) \quad\quad\quad [M]\mapsto [i\text{-}\Res^n_{n-1}M];$$
$$ f_{i,n}:\G(\R_{n-1})\to \G(\R_{n}) \quad\quad\quad [M]\mapsto [i\text{-}\Ind^n_{n-1}M].$$
\begin{defn}We introduce the notation $\G(\R)=\bigoplus_{n\geq 0}G_0(\R_n)$, then we define the following linear operators in $\End(\G(\R))$:
$$e_i=\sum_{n\geq 0}e_{i,n};\quad f_i=\sum_{n\geq 0}f_{i,n};\quad\forall i\in\ZZ\cup\{\infty\}.$$
\end{defn}
\begin{lem}\label{lemcomm}For all $i\in\ZZ$, we have the commutation relations
$$ [e_i,e_\infty]=[e_i,f_\infty]=[f_i,e_\infty]=[f_i,f_\infty]=0.$$
\end{lem}
\begin{proof}This follows immediately from the formulas for the action of $e_i, f_i,e_\infty,f_\infty$ on the simple $\R_n$-modules in Propositions \ref{ires} and \ref{iind}.
\end{proof}
\begin{defn}For $k\geq 0$, let $\F_k=\displaystyle\frac{\ker (e_\infty)^{k+1}}{\ker (e_\infty)^k}$.
\end{defn}
\begin{rem}From Proposition \ref{ires}, we have $\F_k\simeq\CC\{[M^{(\lambda,1^k)}]\}_\lambda\subset\G(\R)$. Thus 
$$\G(\R)\simeq\bigoplus_{k\geq 0}\F_k.$$
\end{rem}
\begin{prop}\label{lemma1}For all $i\in\ZZ$, $e_i(\F_k)\subset \F_k$, $f_i(\F_k)\subset \F_k$. 
In addition, $f_\infty:\F_k\to\F_{k+1}$ is an isomorphism, and $e_\infty=(f_\infty)^{-1}:\F_{k+1}\to\F_{k}$ for all $k\geq 0$. 

For all $k\geq 0$, $\F_k$ is isomorphic to the (charge zero) Fock space as a representation of $\mathfrak{gl}_\infty$, with the action given by  $\{e_i,f_i|i\in\ZZ\}$.

Hence, as $\mathfrak{gl}_\infty$-representations, $\G(\R)$ is the direct sum of countably many copies of the Fock space.
\end{prop}
\begin{proof}The space $\F_k$ is invariant under $\{e_i, f_i\}_{i\in\ZZ}$ by Lemma \ref{lemcomm}. The result then follows directly from the formulas for the action of $e_i,f_i$ of Propositions \ref{ires} and \ref{iind}. In fact, for each fixed $\F_k$ those are exactly the formulas that define the action of the Chevalley generators of $\mathfrak{gl}_\infty$ on the basis of partitions 
of the Fock space. This is the same as what is proved by Ariki in \cite[Thm 4.4]{Ar3}.
\end{proof}
\section{Cellular Structure}\label{cellstruct}
In all the previous sections we have focused on the generic mirabolic Hecke algebra $\R_n$ and its semisimple specializations.
\begin{lem}$\R_n(q)$ is semisimple if and only if $[n]_q!\neq0$. ( $[n]_q!$ was defined in the statement of Theorem \ref{ariki}).
\end{lem} 
\begin{proof}If $[n]_q!\neq 0$, $H_n(1,0;q)$ is semisimple by Theorem \ref{ariki}, and therefore so is $\R_n(q)$ by Proposition \ref{isoiso}. When $[n]_q!=0$, then the Hecke algebra $\HH_n(q)$ is not semisimple. Notice that we have a surjective map $\R_n(q)\ronto \HH_n(q)$ given by $e\mapsto 1$. Then in this case $\R_n(q)$ also cannot be semisimple since it has a non-semisimple quotient.\end{proof} 
To study what happens in the non-semisimple case, it would be extremely helpful to establish that $\R_n(q)$ is a cellular algebra in the sense of Graham and Lehrer (see \cite{GL}). Then the general theory of cellular algebras could be used to describe the representation theory of $\R_n(q)$. 

We give a reminder of the definition of a cellular algebra and a cellular basis for $H(1,0;q)$, following \cite{DJM}.
\begin{defn}\label{cellular}Let $A$ be an algebra $A$ over a ring $\kk$, with a $\kk$-basis
$$ \{C^{\lambda}_{\st,\tu}|\lambda\in\Lambda, \st,\tu\in M(\lambda)\} $$
where $\Lambda$ a partially ordered set, and $M(\lambda)$ a finite set for each $\lambda\in\Lambda$.
We say that $A$ is a \emph{cellular algebra} (and we call the basis we just defined a \emph{cellular basis}) if it satisfies the following properties:
\begin{itemize}
\item[$(C1)$] The $\kk$-linear involution $*:A\to A$ given by $\left(C^\lambda_{\st,\tu}\right)^*=C^\lambda_{\tu,\st}$ is an anti-isomorphism of algebras.
\item[$(C2)$] For each $a\in A$ we have
$$ a C^\lambda_{\st,\tu}\equiv \sum_{\st'\in M(\lambda)} r_a(\st',\st)C^\lambda_{\st',\tu} \quad\pmod {A(>\lambda)}$$
where $ r_a(\st',\st)$ is independent of $\tu$ and $A(>\lambda)$ is the $\kk$-submodule of $A$ generated by $\{C^\mu_{\pp,\qq}|\mu>\lambda,\quad \pp,\qq\in M(\mu)\}$.
\end{itemize}
By applying the antiinvolution $*$, we can also substitute $(C2)$ with the equivalent condition
\begin{itemize}
\item[$(C2')$] For each $a\in A$ we have
$$  C^\lambda_{\st,\tu}a\equiv \sum_{\tu'\in M(\lambda)}C^\lambda_{\st,\tu'} r_a(\tu',\tu) \quad\pmod {A(>\lambda)}$$
where $ r_a(\tu',\tu)$ is independent of $\tu$.
\end{itemize}
In the case of the cyclotomic algebra $H_n(1,0;q)$, we take $\Lambda=\Lambda(n)=\{(\lambda^1,\lambda^2)\}$ the set of all bipartitions of $n$, with the \emph{dominance} partial order $\trianglelefteq$ given by
$$ \ull\trianglelefteq \umu\quad\text{ if }\quad \sum_{i=1}^{k-1}|\lambda^i|+\sum_{i=1}^j \lambda^k_i\leq  \sum_{i=1}^{k-1}|\mu^i|+\sum_{i=1}^j \mu^k_i$$
for all $1\leq k\leq 2$, $j\geq 0$. If $\ull\trianglelefteq\umu$ and $\ull\neq\umu$ we write $\ull\triangleleft\umu$.

We then take $M(\ull)=\T(\ull)$ to be the set of all standard bitableaux of shape $\ull$.

To define the cellular basis we need to define some elements. For a bipartition $\ull$, we define the \emph{superstandard} bitableau $\mathfrak{t}^{\ull}$ to be the standard bitableau of shape $\ull$ filled with the numbers $1,\ldots,n$ in order starting from the first row of the first tableau and going down the rows, then filling the rows of the second tableau. 
\begin{exa}If $\ull=((2,2),(3,1,1))$ then
$$\mathfrak{t}^{\ull}=\left( \young(12,34)\text{ }, \young(567,8,9)\right).$$
\end{exa}
We then have the Young subgroup $S_{\ull}=S_{\lambda^1}\times S_{\lambda^2}$ of the symmetric group $S_n$ which is the stabilizer of the rows of $\ttl$. We define 
$$x_{\ull}=\sum_{w\in S_{\ull}}T_w;\qquad u_{\ull}=\prod_{i=1}^{|\lambda^1|}\tilde{L}_i;\qquad m_{\ull}=x_{\ull}u_{\ull}=u_{\ull}x_{\ull}.$$
For a standard bitableau $\st$, define $d(\st)\in S_n$ to be the permutation such that 
$$\st=d(\st)\ttl.$$
\end{defn}
\begin{thm}[{\cite[Thm 3.26]{DJM}}] A cellular basis for $H_n(1,0;q)$ is given by
$$ C^{\ull}_{\st,\tu}:= \left(T_{d(\st)}\right)^*m_{\ull}T_{d(\tu)}$$
with $*$ being the anti-involution that fixes the generators $X,T_1,\ldots,T_{n-1}$.
\end{thm}
We consider the elements $ {\bar{C}^{\ull}_{\st,\tu}}\in\R_n(q)$ to be the images of the cellular basis of $H_n(1,0;q)$ in the quotient. Let $\Lambda'(n)\subset\Lambda(n)$ be the subset
$$ \Lambda'(n)=\{(\lambda^1,\lambda^2)\in\Lambda(n)|\lambda^2\text{ is a single column}\}$$
with the induced partial order from $\Lambda(n)$.
\begin{conj}\label{cellconj}The elements 
$$ \{ {\bar{C}^{\ull}_{\st,\tu}}|\ull\in\Lambda'(n)\}$$
are a cellular basis for $\R_n(q)$, which is then a cellular algebra.
\end{conj}
The strategy to prove the conjecture is to show the following:
\begin{conj}\label{basisconj} Each element $ {\bar{C}^{\ull}_{\st,\tu}}$, with $\ull\in\Lambda(n)\setminus\Lambda'(n)$ is a $\CC$-linear combination of elements $ {\bar{C}^{\umu}_{\pp,\qq}}$ with $\ull\vartriangleleft\umu$.
\end{conj}
\begin{prop}Conjecture \ref{cellconj} follows from Conjecture \ref{basisconj}.
\end{prop}
\begin{proof}Suppose Conjecture \ref{basisconj} holds. Clearly the elements $\{ {\bar{C}^{\ull}_{\st,\tu}}|\ull\in\Lambda(n)\}$ span $R_n(q)$ because they are the image of a basis under a surjective map. By Conjecture \ref{basisconj}, consider an element ${\bar{C}^{\ull}_{\st,\tu}}$ with $\ull\in\Lambda(n)\setminus\Lambda'(n)$, then it is a linear combination of elements $ {\bar{C}^{\umu}_{\pp,\qq}}$ with $\ull\vartriangleleft\umu$. For each of those elements ${\bar{C}^{\umu}_{\pp,\qq}}$ with $\umu\in\Lambda(n)\setminus\Lambda'(n)$ we can do the same thing and express them as a linear combination of elements strictly higher in the partial order and so on. Since $\Lambda(n)$ is a finite set, this procedure terminates and hence every element of the set $\{ {\bar{C}^{\ull}_{\st,\tu}}|\ull\in\Lambda(n)\}$ is in the span of $\{ {\bar{C}^{\ull}_{\st,\tu}}|\ull\in\Lambda'(n)\}$.
This shows that the latter set spans $\R_n(q)$, hence by a dimension count it is a basis. 

Then, the anti-involution $*:H_n(1,0;q)\to H_n(1,0;q)$ fixing the generators descends to the anti-involution of $\R_n(q)$ of Remark \ref{antiinv}, so condition $(C1)$ in Definition \ref{cellular} is satisfied automatically. Finally, since condition $(C2)$ holds for the cellular basis of $H_n(1,0;q)$, it still holds in the quotient for the elements $\{ {\bar{C}^{\ull}_{\st,\tu}}|\ull\in\Lambda(n)\}$. By Conjecture \ref{basisconj}, we can then reexpress the RHS of $(C2)$ in terms of only elements of $\{ {\bar{C}^{\ull}_{\st,\tu}}|\ull\in\Lambda'(n)\}$, up to terms higher in the partial order. Therefore condition $(C2)$ holds in $\R_n(q)$ for the basis $\{ {\bar{C}^{\ull}_{\st,\tu}}|\ull\in\Lambda'(n)\}$ and Conjecture \ref{cellconj} follows.
\end{proof}
We can easily check Conjecture \ref{basisconj} for $\R_2(q)$.
\begin{thm}$\R_2(q)$ is a cellular algebra.
\end{thm}
\begin{proof}For $n=2$, we have the five bipartitions of $2$ in order
$$(2,\oo)\vartriangleright (11,\oo)\vartriangleright (1,1)\vartriangleright (\oo, 2)\vartriangleright (\oo,11).$$
For four of those, there is only one standard bitableau of the given shape $\ull$, which is the superstandard one $\ttl$. The exception is $(1,1)$ for which we have the two bitableaux
$$\mathfrak{t}^{(1,1)}= \left(\young(1),\young(2)\right);\quad\mathfrak{s}^{(1,1)}=\left(\young(2),\young(1)\right).$$
We have 
\begin{center}
\begin{tabular}{|c|c|c|}\hline
$\ull$ & $x_{\ull}$ & $u_{\ull}$ \\ \hline
$(2,\oo)$ & $1+T_1$ & $L_1L_2$ \\ \hline
$(11,\oo)$ & $1$ & $L_1L_2$ \\ \hline
$(1,1)$ & $1$ & $L_1$ \\ \hline
$(\oo,2)$ & $1+T_1$ & $1$ \\ \hline
$(\oo,11)$ & $1$ & $1$ \\ \hline
\end{tabular}
\end{center}
The cellular basis of $H_n(1,0;q)$ is then given by
\begin{align*} C^{(2,\oo)}_{\mathfrak{t},\mathfrak{t}}&= (1+T_1)L_1L_2 \\
C^{(11,\oo)}_{\mathfrak{t},\mathfrak{t}}&= L_1L_2 \\
C^{(1,1)}_{\mathfrak{t},\mathfrak{t}}&= L_1 \\
C^{(1,1)}_{\mathfrak{t},\mathfrak{s}}&= L_1T_1 \\
C^{(1,1)}_{\mathfrak{s},\mathfrak{t}}&= T_1L_1 \\
C^{(1,1)}_{\mathfrak{s},\mathfrak{s}}&= T_1L_1T_1 \\
C^{(\oo,2)}_{\mathfrak{t},\mathfrak{t}}&= 1+T_1 \\
C^{(\oo,11)}_{\mathfrak{t},\mathfrak{t}}&= 1.
\end{align*}
A direct computation shows that in $\R_2(q)$, where $L_1=e$ and $L_1L_2=q^{-1}T_1eT_1e$, we have
$$ 1+T_1=T_1L_1T_1+T_1L_1+L_1T_1+L_1-(1+T_1)L_1L_2 $$
hence
$$\bar{C}^{(\oo,2)}_{\mathfrak{t},\mathfrak{t}}=\bar{C}^{(1,1)}_{\mathfrak{s},\mathfrak{s}}+
\bar{C}^{(1,1)}_{\mathfrak{s},\mathfrak{t}}+\bar{C}^{(1,1)}_{\mathfrak{t},\mathfrak{s}}+
\bar{C}^{(1,1)}_{\mathfrak{t},\mathfrak{t}}-\bar{C}^{(2,\oo)}_{\mathfrak{t},\mathfrak{t}}.$$
So indeed, since $(\oo,2)$ is the only partition not in $\Lambda'$ and since $(\oo,2)\vartriangleleft (1,1)$, $(\oo,2)\vartriangleleft (2,\oo)$ we have that Conjecture \ref{basisconj} holds for $\R_2(q)$, hence by Conjecture \ref{cellconj} the theorem follows.
\end{proof}
To prove Conjecture \ref{basisconj} in the case of $n\geq 3$, it is impractical to check directly that it holds for all the elements $\bar{C}^{\ull}_{\st,\tu}$, $\ull\in\Lambda(n)\setminus\Lambda'(n)$, like we did in the case of $n=2$. However, we can greatly reduce the number of basis elements that we need to check.
\begin{prop}For a fixed bipartition $\ull\in\Lambda(n)\setminus\Lambda'(n)$, if Conjecture \ref{basisconj} holds for $m_{\ull}=\bar{C}^{\ull}_{\ttl,\ttl}$, where $\ttl$ is the superstandard tableau, then it holds for all $\bar{C}^{\ull}_{\st,\tu}$, $\st,\tu\in M(\ull)$.
\end{prop}
\begin{proof}Suppose 
$$\bar{C}^{\ull}_{\ttl,\ttl}=m_{\ull}=\sum_{{\ull}\triangleleft{\umu}}\gamma_{\pp,\qq} \bar{C}^{\umu}_{\pp,\qq};$$
then for any $a\in\R_n(q)$ we have
$$a m_{\ull}=\sum_{{\ull}\triangleleft{\umu}}\gamma_{\pp,\qq} a\bar{C}^{\umu}_{\pp,\qq}.$$
But we know from the condition $(C2)$ in $H(1,0;q)$ that each $ a\bar{C}^{\umu}_{\pp,\qq}$ is a linear combination of elements $\bar{C}^{\unu}_{\mathfrak{r},\mathfrak{n}}$ with $\umu\trianglelefteq\unu$, hence the same is true for $a m_{\ull}$. We can then apply condition $(C2')$ to $a m_{\ull} b$ to obtain that it has to be a linear combination of elements as required.

To conclude we just need to remember that, by definition, $ C^{\ull}_{\st,\tu}:= \left(T_{d(\st)}\right)^*m_{\ull}T_{d(\tu)}$.
\end{proof}
Using the following simple observation, we can then aim to set up an inductive argument.
\begin{lem}\label{remm}If $\ull=(\lambda^1,\lambda^2)$ is a bipartition of $n$ and $\ull'=(\lambda^1,(\lambda^2,1))$ is the bipartition of $n+1$ obtained by adding a $1$ at the end of $\lambda^2$ then, under the inclusion $H_n(1,0;q)\rinto H_{n+1}(1,0;q)$, we have
$$m_{\ull}\mapsto m_{\ull'}.$$
\end{lem}
\begin{proof}The Young subgroup $S_{\ull}=S_{\lambda^1}\times S_{\lambda^2}$, under the inclusion $S_n\rinto S_{n+1}$ is mapped to $S_{\lambda^1}\times S_{\lambda^2}\times S_1=S_{\ull'}\simeq S_{\lambda^1}\times S_{\lambda^2}$. Hence $x_{\ull}=x_{\ull'}$. By its definition, the term $u_{\ull}$ only depends on $\lambda^1$, so it is also identified with $u_{\ull'}$.
\end{proof}
We have a map
$$\varphi:\Lambda(n)\setminus\Lambda'(n)\to \Lambda(n+1)\setminus\Lambda'(n+1)$$
given by
$$(\lambda^1,\lambda^2)\mapsto (\lambda^1,(\lambda^2,1)).$$
By Lemma \ref{remm}, if Conjecture \ref{basisconj} holds for a bipartition $\ull\in\Lambda(n)\setminus\Lambda'(n)$, then it holds for $\varphi(\ull)$. We can then hope to build up to a proof of the conjecture for all $n$ starting from the case $n=2$.
Of course it is not that easy because the map $\varphi$ is very far from being surjective, but at least it gives us a starting point.
For example for $n=3$, we have that $\Lambda(3)\setminus\Lambda'(3)=\{(1,2),(\oo,3),(\oo,21)\}$. Of those three cases, we have $(\oo,21)=\varphi(\oo,2)$, but to prove that $\R_3(q)$ is cellular we would still need to find appropriate expressions for $m_{(1,2)}$ and $m_{(\oo,3)}$ as linear combinations.
\begin{rem}Travkin, in his paper \cite{T}, has defined a special basis of $\R_n(q)$ using its interpretation as a convolution algebra as in the beginning of Section \ref{mirha}. The definition uses the geometry of the $G$-orbits on $G/B\times G/B\times V$. Travkin calls it the Kazhdan-Lusztig basis of $\R_n(q)$ and, by definition, it is invariant under a KL-involution of $\R_n(q)$, and can be used to define left, right and two-sided KL cells. In \cite{T}, it is conjectured that these cells are related to the mirabolic RSK correspondence that is defined in the paper. The problem with that construction, which makes it not applicable to the conjectures of this section, is that it only uses the $\HH_n(q)$-bimodule structure of $\R_n(q)$ and not its algebra structure. In fact the KL-involution defined by Travkin is not an algebra involution.
Finding a cellular basis for $\R_n(q)$ (and possibly relating it to the mirabolic RSK correspondence) could then also be interpreted as a way of restating and answering Travkin's conjecture.
\end{rem}


\begin{thebibliography}{99999}
\bibitem[AH]{AH} P. N. Achar, A. Henderson: \emph{Orbit closures in the enhanced nilpotent cone}, Adv. in Math. \textbf{219} (2008), 27-62.
\bibitem[Ar1]{Ar3} S. Ariki: \emph{On the decomposition numbers of the Hecke algebra of G(m,1,n)}. J. Math. Kyoto Univ. \textbf{36} (1996), no. 4, 789-808.
\bibitem[Ar2]{Ar} S. Ariki: \emph{On the semi-simplicity of the Hecke algebra of $(\ZZ/r\ZZ)\wr \Sigma_n$}. J. Algebra \textbf{169} (1994), no. 1, 216-225.
\bibitem[Ar3]{Ar2}S. Ariki:\emph{Representations of quantum algebras and combinatorics of Young tableaux}. University Lecture Series, 26. American Mathematical Society, Providence, RI, 2002. viii+158 pp.
\bibitem[AK]{AK} S. Ariki, K. Koike: \emph{A Hecke algebra of $(\ZZ/r\ZZ)\wr \Sigma_n$ and construction of its irreducible representations}. Adv. Math. \textbf{106} (1994), no. 2, 216-243. 
\bibitem[DJ]{DJ} R. Dipper, G. James: \emph{Blocks and idempotents of Hecke algebras of general linear groups}. Proc. London Math. Soc. (3) \textbf{54} (1987), no. 1, 57-82.
\bibitem[DJM]{DJM}  R. Dipper, G. James, A. Mathas: \emph{Cyclotomic q-Schur algebras}. Math. Z. \textbf{229} (1998), no. 3, 385-416.
\bibitem[FG]{FG} M. Finkelberg, V. Ginzburg: \emph{Cherednik algebras for algebraic curves}. Representation theory of algebraic groups and quantum groups, 121-153, Progr. Math., \textbf{284}, Birkh\"{a}user/Springer, New York, (2010).
\bibitem[GL]{GL}  J.J. Graham, G.I. Lehrer: \emph{Cellular algebras}. Invent. Math. \textbf{123} (1996), no. 1, 1-34.
\bibitem[Ha]{Ha}T. Halverson: \emph{Representations of the q-rook monoid}, J. Algebra \textbf{273} (2004), no. 1, 227-251. 
\bibitem[HR]{HR} T. Halverson, A.Ram: \emph{$q$-rook monoid algebras, Hecke algebras, and Schur-Weyl duality}, Zap. Nauchn. Sem. S.-Peterburg. Otdel. Mat. Inst. Steklov. (POMI) \textbf{283} (2001), Teor. Predst. Din. Sist. Komb. i Algoritm. Metody. 6, 224-250, 262-263.
\bibitem[Ho]{Ho}P.N. Hoefsmit: \emph{Representations of Hecke algebras of finite groups with BN-pairs of classical type}, Ph.D. Thesis - The University of British Columbia (Canada). 1974.
\bibitem[J]{J}  Jucys, A.-A. A. : \emph{Symmetric polynomials and the center of the symmetric group ring}, Rep. Mathematical Phys. \textbf{5} (1974), no. 1, 107-112.
\bibitem[K]{K} S. Kato: \emph{An exotic Deligne-Langlands correspondence for symplectic groups}, Duke Math. J. \textbf{148} (2009), no. 2, 305-371.
\bibitem[MWZ]{MWZ}  P. Magyar, J. Weyman, A. Zelevinsky: \emph{Multiple flags of finite type}, Adv. Math. \textbf{141} (1999), 97-118.
\bibitem[Mu]{Mu} Murphy, G. E. : \emph{The idempotents of the symmetric group and Nakayama's conjecture}, J. Algebra \textbf{81} (1983), no. 1, 258-265. 
\bibitem[OV]{OV} Okounkov, Andrei; Vershik, Anatoly: \emph{A new approach to representation theory of symmetric groups}. Selecta Math. (N.S.) \textbf{2} (1996), no. 4, 581-605.
\bibitem[Ra]{Ra}A. Ram: \emph{Seminormal representations of Weyl groups and Iwahori-Hecke algebras}, Proc. London Math. Soc. (3) \textbf{75} (1997), no. 1, 99-133.
\bibitem[RR]{RR}A. Ram, J. Ramagge: \emph{Affine Hecke algebras, cyclotomic Hecke algebras and Clifford theory}, A tribute to C. S. Seshadri (Chennai, 2002), 428-466, Trends Math., Birkh\"auser, Basel, 2003
\bibitem[Si]{Si} E. A. Siegel: \emph{The representations of a Hecke algebra of the affine group over a finite field}, J. Algebra \textbf{159} (1993), no. 2, 515-539.
\bibitem[So1]{So} L. Solomon: \emph{The affine group. I. Bruhat decomposition}, J. Algebra \textbf{20} (1972), 512-539.
\bibitem[So2]{So2}L. Solomon:\emph{The Bruhat decomposition, Tits system and Iwahori ring for the monoid of matrices over a finite field}, Geom. Dedicata \textbf{36} (1990), no. 1, 15-49.
\bibitem[T]{T} R. Travkin: \emph{Mirabolic Robinson-Schensted-Knuth correspondence}, Sel. Math. (N.S.) \textbf{14} (2009), 727-758.
\end{thebibliography}
\end{document}